\documentclass{sn-jnl}
\usepackage{hyperref}
\usepackage{amssymb}
\usepackage{amsmath}
\usepackage{amsthm}

\usepackage{tikz-cd}

\newtheorem{theorem}{Theorem}[section]
\newtheorem{proposition}[theorem]{Proposition}%
\newtheorem{fact}[theorem]{Fact}
\newtheorem{lemma}[theorem]{Lemma}%
\newtheorem{corollary}[theorem]{Corollary}%

\begin{document}

\title[Article Title]{Enumeration formulas for (3, 6)-fullerenes}

\newcommand{\orcidA}{0000-0001-7066-4353}

\author*[1]{\fnm{Linda} \sur{Green}}\email{greenl@email.unc.edu}

\author[1]{\fnm{Yadunand} \sur{Sreelesh}}\email{yadu21@unc.edu}

\author[1]{\fnm{Saanvi} \sur{Arora}}\email{saanvi@email.unc.edu}

\affil[1]{\orgname{University of North Carolina at Chapel Hill}, \orgaddress{\city{Chapel Hill}, \state{NC},  \country{USA}}}

\abstract{A $(3, 6)$-fullerene is a cubic planar graph whose faces all have 3 or 6 sides. We give an exact enumeration of $(3, 6)$-fullerenes with $V$ vertices. We also enumerate $(3, 6)$-fullerenes with mirror symmetry, with 3-fold rotational symmetry, and with both types of symmetry. The resulting formulas are expressed in terms of the prime factorization of $V$. }


\maketitle

\section{Introduction}

\label{sec-intro}

A \emph{$(3, 6)$-fullerene} is a connected, cubic, planar graph whose faces all have three or six sides, with a fixed planar embedding, up to orientation-preserving homeomorphisms of the plane. We will say that two  $(3, 6)$-fullerenes are equivalent if they are not only isomorphic as graphs but if there is also an orientation-preserving homeomorphism of the plane that is a graph isomorphism. Thus, left-handed and right-handed $(3,6)$-fullerenes are considered distinct.  We will sometimes consider a $(3,6)$-fullerene to be embedded in the sphere instead of the plane, using the one-point compactification of the plane. 
 We refer to faces with three sides as  ``triangles'' and faces with six sides as ``hexagons'', even though these faces may not have straight edges and may be unbounded.  Note that the assumption that all faces are triangles or hexagons ensures that the graphs are simple with no edge-loops. 
 
 All $(3,6)$-fullerenes are either 2-connected or 3-connected \cite{yang2012hexagonal}. The only $(3,6)$-fullerenes that are not 3-connected are those whose planar embedding is composed of two triangles, surrounded by one or more nested pairs of hexagons that meet along opposite sides, with a final pair of adjacent triangles on the outside \cite{deza2005zigzag, yang2012hexagonal}.  These 2-connected planar graphs are denoted by $G_n$ in Section 4 of \cite{deza2005zigzag} and referred to as ``godseyes'' in Section 2 of \cite{green2024polyhedra}.  See Figure~\ref{fig-godseye}. In our counts of $(3,6)$-fullerenes, we include godseyes as well as 3-connected graphs, because the proofs and count formulas are simpler using that convention. Some other sources, for example Deza and Detour \cite{deza2005zigzag}, exclude godseyes in their counts. This difference in convention results in our counts being one higher than those of Deza and Detour \cite{deza2005zigzag}, as explained further in Section~\ref{sec-examples}.

 \begin{figure}[ht]
       \centering
    \includegraphics[height = 5 cm]{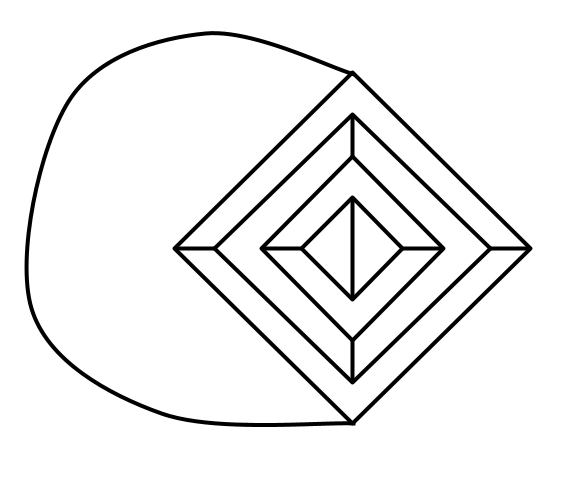}
    \caption{A godseye is the only type of $(3,6)$-fullerene that is not 3-connected. }
\label{fig-godseye}
\end{figure}

 The $(5,6)$-fullerenes, whose faces are all pentagons and hexagons, have been widely studied as mathematical representations of carbon molecules. P. Engle and P. Smillie \cite{engel2025exact} gave explicit formulas for the number of $(5,6)$-fullerenes in terms of the number of vertices. The (3,6)–fullerenes have also received attention due to their similarity to $(5,6)$-fullerenes, by B. Gr\"unbaum and T.S. Motzkin \cite{grunbaum_motzkin_1963}, P.W. Fowler, P.E. John, and H. Sachs \cite{fowler19983} and \cite{JOHN20092663}, M. Deza and M. Dutour \cite{deza2005zigzag},  M. DeVos, L. Goddyn, B. Mohar, and R. \u{S}\'{a}mal \cite{DEVOS2009358}, and others. 
Efficient algorithms exist for generating $(3,6)$-fullerenes and thereby establishing exact counts of (3, 6)-fullerenes for any number of vertices  \cite{brinkmann1997constructive}, \cite{brinkmann2010cage}  \cite{brinkmann2012generation}, \cite{brinkmann2003construction}, and \cite{brinkmann2007fast}. Counts for up to 512 vertices are given in Table 5 of \cite{deza2005zigzag}. This paper extends previous work by providing formulas for the number of $(3,6)$-fullerenes with $V$ vertices in terms of the prime factorization of $V$. 

This work is related to work by J. Huijben \cite{huijben2016tetrahedral}, who independently found versions of formulas  in Theorem~\ref{thm-delta}, Theorem~\ref{prop-mu}, and Lemma~\ref{lem-$(3,6)$-fullereneGraphCount} of the current paper. However, the current work is novel in its approach, because it uses a classification of $(3,6)$-fullerenes in terms of triples of numbers, which allows symmetry properties to be translated into simple equations in modular arithmetic. Unlike Huijben’s thesis, which uses Eisenstein integers and ideal counting for tetrahedral triangulations, the current work uses a more elementary approach that involves only basic combinatorial and arithmetic arguments. The current work gives explicit formulas for counts in places where Huijben's thesis gives recursive formulas.

 In this paper, Section~\ref{sec-background} reviews background information from \cite{green2024polyhedra}.  It describes how each $(3,6)$-fullerene can be described by a triple of numbers, called a \emph{signature}. Each $(3,6)$-fullerene has three signatures, and these three signatures are either all distinct, or they are all the same, in which case they are called \emph{coinciding signatures}.  Section~\ref{sec-symmetry} establishes that $(3,6)$-fullerenes have coinciding signatures if and only they are equivalent to embeddings on a sphere with 3-fold rotational symmetry. Section~\ref{sec-equation}  converts the problem of finding $(3,6)$-fullerenes with coinciding signatures to the problem of finding solutions to the equation $x^2 + x + 1 \equiv 0 \pmod{n}$, where $n$ is a factor of the number of vertices.  Section~\ref{sec-(3,6)-fullereneCounts} uses the number of solutions to the equation $x^2 + x + 1 \equiv 0 \pmod n$ to find the number of $(3,6)$-fullerenes with $V$ vertices, as well as the number with 3-fold rotational symmetry. Section~\ref{sec-mirror} gives equivalent conditions for a $(3,6)$-fullerene to have mirror symmetry. Section~\ref{sec-mirrorCounts} gives formulas to count the number of $(3,6)$-fullerenes with $V$ vertices with mirror symmetry, as well as the number with both mirror and 3-fold rotational symmetry. Section~\ref{sec-(3,6)-fullereneGraphs} gives the number of graph isomorphism classes of $(3,6)$-fullerenes with $V$ vertices. The number of graph isomorphism classes is smaller than the number of $(3,6)$-fullerenes since two $(3,6)$-fullerenes can be isomorphic as graphs but not equivalent to each other if they are mirror images. Section~\ref{sec-examples} offers sample computations and a table of counts. With the exception of parts of Sections \ref{sec-(3,6)-fullereneGraphs} and \ref{sec-examples}, all  enumerations are of oriented embeddings, i.e. chiral map-classes, not just abstract graphs.

 \section{Background and formulas}
 \label{sec-background}
This section reviews background from previous work \cite{green2024polyhedra}, including a topological description of $(3,6)$-fullerenes as a quotient space and an algebraic characterization of $(3,6)$-fullerenes in terms of triples of numbers. 

Consider the quotient space of a hexagonal tiling of the
plane, under a group of isometries generated by $180^\circ$ rotations, where the centers of rotation lie on the centers of hexagons and form the vertices of a superimposed 
parallelogram grid. See Figures~\ref{fig:cover} and \ref{fig:quotient} for examples. The fundamental domain of this group action consists of two adjacent parallelograms, and the quotient space is homeomorphic to a sphere. The edges and vertices of the hexagonal tiling quotient to edges and vertices of a $(3,6)$-fullerene embedded in that sphere.  The hexagons of the hexagonal tiling of the plane that contain centers of rotation are called \emph{special hexagons}. Each special hexagon covers a triangular face of the $(3,6)$-fullerene, and every other hexagon in the hexagonal tiling covers a hexagonal face of the $(3,6)$-fullerene. This quotient space construction is related to construction 4.1 of \cite{DEVOS2009358}, where $(3, 6)$-fullerenes are constructed by folding a triangle whose vertices lie on a triangular grid and taking its dual.   

In fact, every $(3,6)$-fullerene, with its planar embedding extended from the plane to the sphere, is homeomorphic via an orientation-preserving homeomorphism of the sphere to such a quotient space. See Theorem 2 of \cite{green2024polyhedra}, which relies on an analysis by Gr\"{u}nbaum and Motzkin in \cite{grunbaum_motzkin_1963}. 

\begin{figure}[ht]
       \centering
    \includegraphics[height = 5 cm]{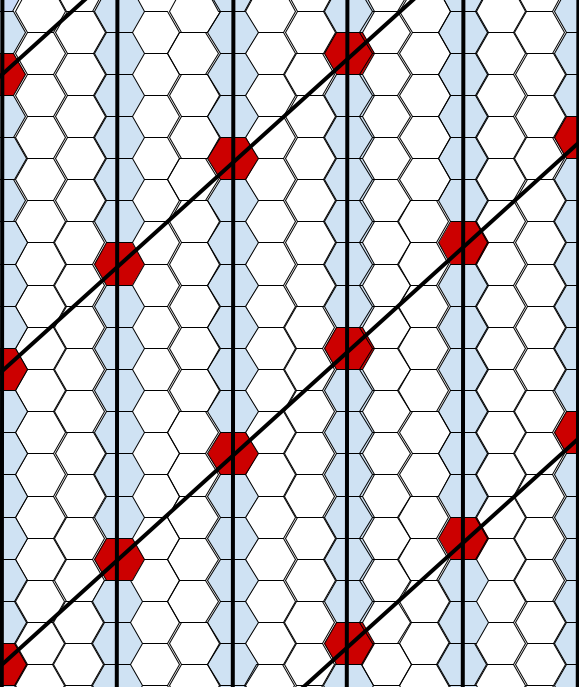}
    \hspace{1 cm}
    \includegraphics[height = 5 cm]{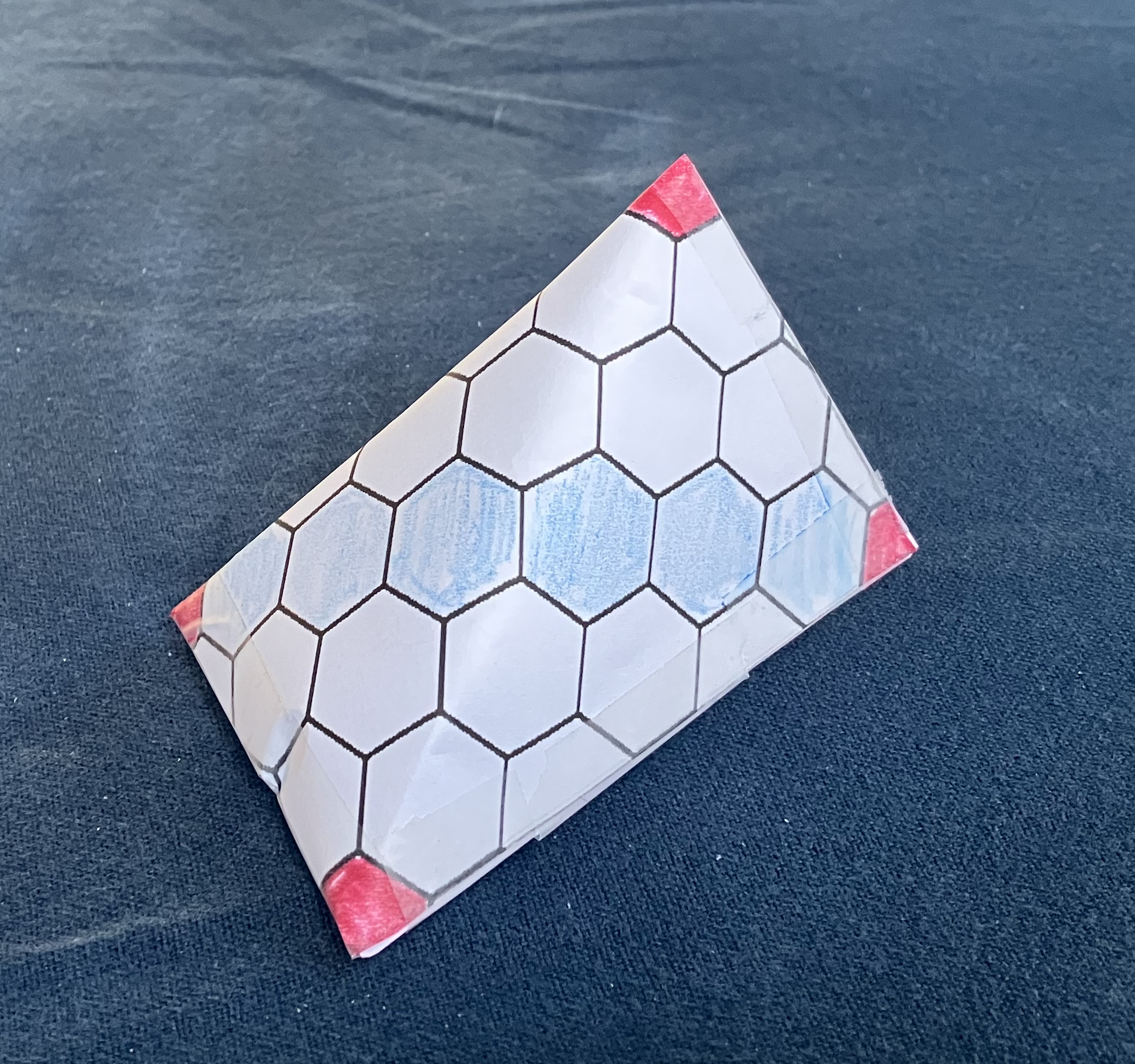}
    \caption{The left figure shows the hexagonal grid that covers a $(3,6)$-fullerene with signature $(6,2,1)$. Special hexagons are shaded dark red (or dark gray if printed in grayscale). The light blue (or light gray) hexagons, together with the special hexagons, form spines of length 6. The vertices of a superimposed parallelogram grid lie in the centers of special hexagons. The right figure shows the $(3,6)$-fullerene formed as the orbifold quotient of the hexagonal tiling.  }
\label{fig:cover}
\end{figure}

\begin{figure}[ht]
\centering
\includegraphics[height = 4 cm]{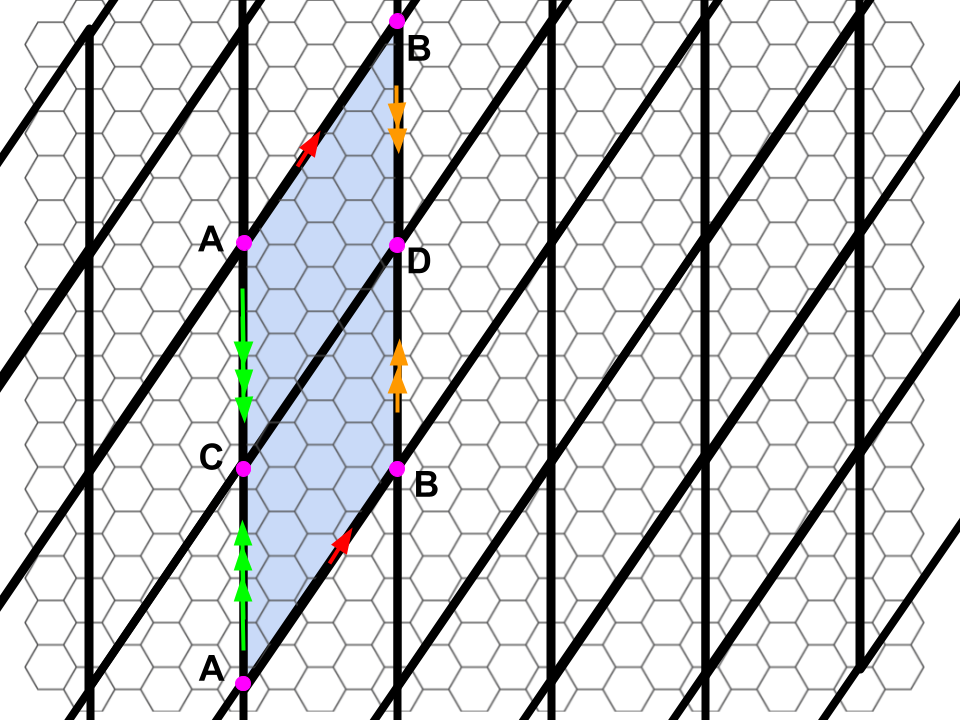} \hspace{0.5 cm} \includegraphics[height = 4 cm]{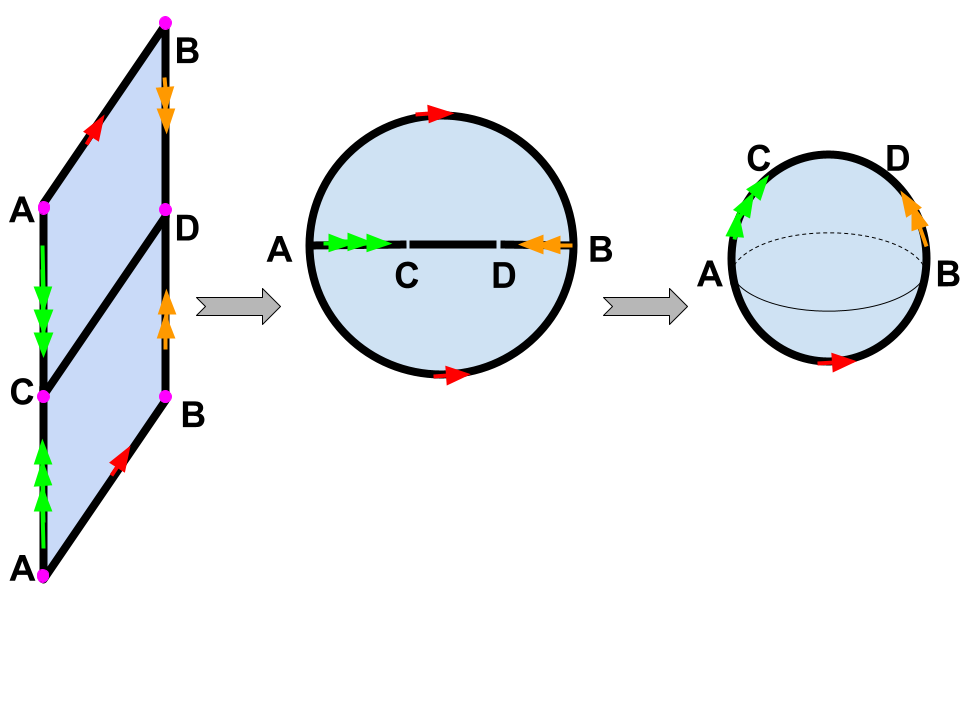}
    \caption{The quotient of a hexagonal tiling under a group of isometries generated by $180^\circ$ rotations around a superimposed parallelogram grid. The quotient space is a topological sphere. The hexagonal grid quotients to a $(3,6)$-fullerene with signature $(4, 3, 3)$ embedded in that sphere.}
    \label{fig:quotient}
\end{figure}

Suppose that the hexagonal tiling is positioned so that the hexagons lie in vertical columns as in Figure~\ref{fig:cover}. We will call the columns that do not contain special hexagons \emph{belt columns}. We will call the columns containing special hexagons \emph{spine columns} and say that two spine columns are consecutive if there are no other spine columns between them.  For example, the hexagonal tiling in Figure~\ref{fig:cover} has two belt columns between consecutive spine columns. A \emph{spine} is a collection of hexagons in a spine column between two special hexagons, including the special hexagons at each end. The length of the spine is the number of hexagons between the two special hexagons, not including the special hexagons on each end. For example, in Figure~\ref{fig:cover}, spines of length 6 are lightly shaded with darkly shaded special hexagons on each end.

A \emph{signature} for a $(3,6)$-fullerene is a triple $(s, b, f)$, where $s \geq 0$ is the length of a spine and $b \geq 0$ is the number of belt columns that lie between two consecutive spine columns. The number $f$, called the offset, defined mod ${s+1}$, indicates how the special hexagons in one spine column are shifted compared to those in a consecutive spine column. Specifically, if we translate a special hexagon along hexagons, $b+1$ columns in the (approximately) southwest to northeast direction, then the offset $f$ is the number of hexagons below a special hexagon that it lands. For example, in 
Figure~\ref{fig:cover}, the offset is 1.

Spines, belts, and offset are defined here in terms of the hexagon tiling of the plane that covers a $(3,6)$-fullerene. But they can also be defined in terms of the $(3,6)$-fullerene itself formed as quotient of this hexagonal tiling. In the $(3,6)$-fullerene itself, the triangles are quotients of the special hexagons at the corners of each parallelogram in the parallelogram grid. The spines are strings of hexagons capped by these triangles, and are the quotients of the spines in the hexagonal tiling. Belts are circuits of hexagons that surround spines and are the quotients of belt columns. Offset can be defined solely in terms of how one spine is rotated with respect to the other in the $(3,6)$-fullerene. See \cite{green2024polyhedra} for details.

The number of triangles in a $(3,6)$-fullerene is always four. This follows from a standard Euler characteristic argument. See, for example, \cite{grunbaum_motzkin_1963}.

Let $h$ be the number of hexagons in the $(3,6)$-fullerene with signature $(s, b, f)$ that is the quotient space of the hexagonal tiling and let $V$ be the number of vertices. To count $h$, note that in the quotient space, there are two spines, each containing $s$ hexagons, and there are $b$ belts, each containing $2(s+1)$ hexagons. Therefore,
\begin{flalign}
    \label{eqn-h}
    h = 2sb + 2s + 2b
\end{flalign} 
To count the number of vertices, note that $V = \dfrac{6h + 3\cdot 4}{3}$, since each hexagon has six vertices, each of the four triangles has three vertices, and there are three faces around each vertex. Therefore, substituting Equation~\ref{eqn-h} for $h$ yields the following equation. 
\begin{flalign}
    \label{eqn-V}
    V = 4(s+1)(b+1)
\end{flalign}
Note that since $s$ and $b$ are integers, $V$ is always a multiple of 4.

Instead of using vertical
strings of hexagons to make spine columns, we could build spine columns by setting off from a special
hexagon at an angle $60^\circ$ clockwise from due north (approximately the southwest (SW) to northeast (NE) direction) or at an angle $120^\circ$ clockwise from due north (approximately northwest (NW) to southeast (SE) direction).  See Figure~\ref{fig:alternativeSpinesUpstairs}. We will also refer to strings of hexagons in these alternate directions as belt columns, spine columns, and spines.

\begin{figure}[ht]
    \centering
    \includegraphics[ height=6cm]{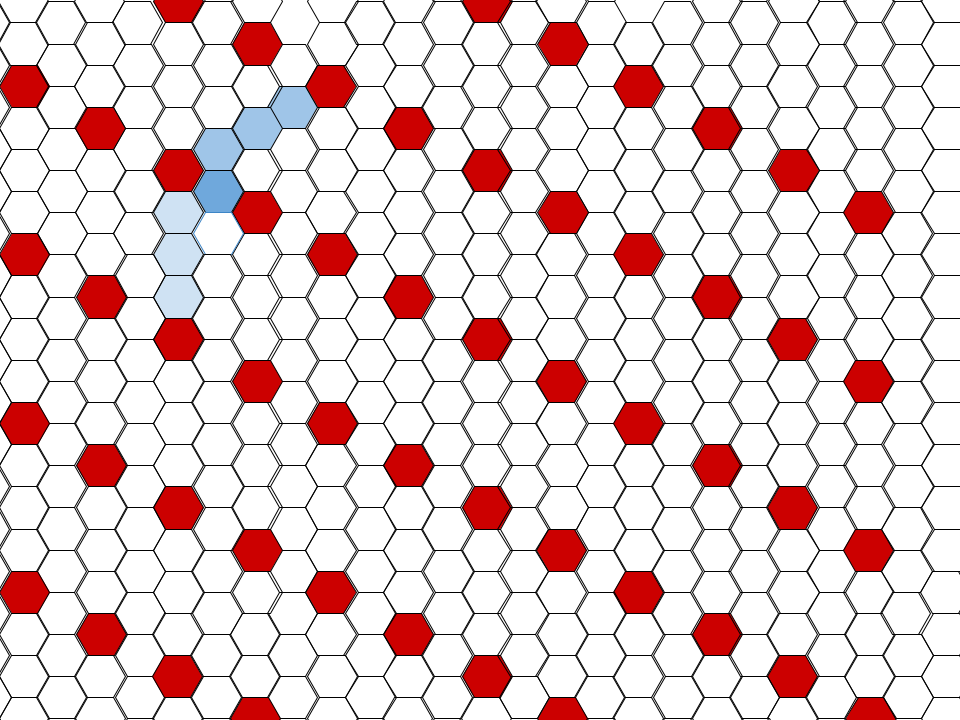}
    \caption{Alternative decompositions into spines in the hexagonal cover. Three spines are shaded in light blue (or light gray if printed in grayscale) with dark red special hexagons on either end (dark gray if printed in grayscale). These spines have lengths 3, 3, and 1. The equivalent signatures for the $(3,6)$-fullerene are $(3, 1, 2)$, $(3, 1, 0)$, and $(1, 3, 0)$}
    \label{fig:alternativeSpinesUpstairs}
\end{figure}

Every $(3,6)$-fullerene has \emph{equivalent signatures}, corresponding to these three directions (North to South, SW to NE, and NW to SE). As proved in \cite{green2024polyhedra}, these equivalent signatures  $(s_1, b_1, f_1)$, $(s_2, b_2, f_2)$, and $(s_3, b_3, f_3)$ are related by the following equations:

\begin{equation}
\label{eqn-s2}
s_2 = j_2(b_1 + 1) - 1
\end{equation}
where $j_2$ is the order of $f_1$ in $\mathbf{Z}_{s_1 + 1}$, that is, the smallest positive integer such that $j_2 \cdot f_1 \equiv 0 \pmod {s_1 +1 }$,

\begin{equation}
\label{eqn-b2}
b_2 = \dfrac{h - 2s_2}{2s_2 + 2}
\end{equation}
 where $h$, the number of hexagons, is given by $h = 2s_1  b_1 + 2s_1 + 2b_1$ by Equation~\ref{eqn-h},

\begin{equation}
\label{eqn-f2}
f_2 \equiv  - p_2(b_1 + 1) - (b_2 +1)\pmod {s_2 + 1}
\end{equation}
where $p_2$ is the smallest positive integer such that  $p_2 \cdot f_1 \equiv (b_2 + 1) \pmod{s_1 + 1}$,
\begin{equation}
\label{eqn-s3}
s_3 = j_3(b_1 + 1) - 1
\end{equation} 
where $j_3$ is the order of $f_1 + b_1 + 1$ in $\mathbf{Z}_{s_1 + 1}$, that is, the smallest positive integer such that $j_3 \cdot(f_1 + b_1 + 1) \equiv 0 \pmod {s_1 + 1}$,
\begin{equation}  
\label{eqn-b3}
b_3 = \dfrac{h - 2s_3}{2s_3 + 2}
\end{equation}
 where $h$, the number of hexagons, is given by $h = 2s_1  b_1 + 2s_1 + 2b_1$ by Equation~\ref{eqn-h},
 and
\begin{equation}
\label{eqn-f3}
f_3 \equiv  - p_3(b_1 + 1) \pmod{s_3 + 1}
\end{equation}
where $p_3$ is the smallest positive integer with $p_3 \cdot (f_1 + b_1 + 1) \equiv (b_3 + 1) \pmod{s_1 + 1}$.

The signatures $(s_1, b_1, f_1)$, $(s_2, b_2, f_2)$, and $(s_3, b_3, f_3)$ give the three alternative descriptions of the same arrangement of special hexagons on a hexagonal tiling of the plane, found by rotating the ``vertical'' direction by 0 or 180 degrees, 60 or 240 degrees, and 120 or 300 degrees clockwise, respectively. Therefore, this definition of equivalent signatures does in fact describe an equivalence relationship.

Recall that two $(3,6)$-fullerenes are equivalent if there is an orientation-preserving homeomorphism of the plane that is an isomorphism of the embedded graphs.  Theorem 4 from \cite{green2024polyhedra}, repeated below, says that equivalent $(3,6)$-fullerenes have equivalent signatures and vice versa:

\begin{theorem}
Suppose $T_0$ and $T_1$ are trihexes with signatures $(s_0, b_0, f_0)$ and $(s_1, b_1, f_1)$. Then $T_0$ and $T_1$ are equivalent trihexes if and only if $(s_0, b_0, f_0)$ and $(s_1, b_1, f_1)$ are equivalent signatures.

Therefore, there is a bijection between equivalence classes of trihexes and equivalence classes of signatures.
\end{theorem}

The proof of this theorem is given in \cite{green2024polyhedra}, but sketched here for completeness. It is clear that two $(3,6)$-fullerenes with equivalent signatures are equivalent, since they are each homeomorphic via an orientation preserving homeomorphism to the same quotient space, where the quotient space is simply described in two different ways by rotating the plane. Conversely, any two equivalent $(3,6)$-fullerenes have equivalent signatures, by the following argument. Any orientation-preserving homeomorphism between two $(3, 6)$-fullerenes induces an orientation-preserving homeomorphism between the two quotient spaces that give their signatures. This homeomorphism can be lifted to an orientation-preserving  homeomorphism between the hexagonally tiled planes that cover them, that preserves special hexagons. This homeomorphism is homotopic to a unique isometry of the plane, that agrees with the original homeomorphism on all vertices of hexagons. Because it is orientation-preserving, the only possibility for such an isometry is a translation or a rotation by a multiple of $60^\circ$. A translation or a rotation by $0^\circ$ or $180^\circ$ takes vertical spines to vertical spines. A rotation by $60^\circ$ or $240^\circ$ counterclockwise takes vertical spines to NW-to-SE spines, and a rotation by $120^\circ$ or $300^\circ$ counterclockwise takes vertical spines to SW-to-NE spines. Therefore, the two signatures must be equivalent. 

\vspace{0.3 cm}

In some cases, equations~\ref{eqn-s2}-\ref{eqn-f3}  generate \emph{coinciding signatures}: that is, the triples of numbers $(s_1, b_1, f_1)$, $(s_2, b_2, f_2)$, and $(s_3, b_3, f_3)$ are exactly the same. For example, the $(3,6)$-fullerene with signature $(13, 1, 4)$ has coinciding signatures.

  \section{Coinciding signatures and rotational symmetry}

\label{sec-symmetry}

We will say that a $(3,6)$-fullerene has \emph{3-fold rotational symmetry} if it is equivalent to an embedding on the sphere in which a 3-fold rotation of the sphere induces a graph isomorphism. This section shows that the $(3,6)$-fullerenes with coinciding signatures are exactly the $(3,6)$-fullerenes that have 3-fold rotational symmetry.

Recall from Section~\ref{sec-background} that every $(3,6)$-fullerene is equivalent to the quotient space of a hexagonal tiling of the plane, under a group of isometries generated by $180^\circ$ rotations.

\pagebreak
\begin{lemma}
\label{lem-welldefined}
Consider a $(3,6)$-fullerene that is represented as the quotient space $Q$ of the hexagonally tiled plane under a group of isometries generated by $180^\circ$ rotations around points of a superimposed parallelogram grid. Let $q: \mathbb{R}^2 \to Q$ be the quotient map. Suppose that $t$ is an isometry of the plane that takes special hexagons to special hexagons. Then the induced map $f:Q \rightarrow Q$ defined by $f(x) = q(t(y))$, where $y$ is any point in $q^{-1}(x)$, is a well-defined map that induces a graph isomorphism on the vertices and edges of $Q$. 
    \end{lemma}

    \[
\begin{tikzcd}
\mathbb{R}^2 \arrow[r, "t"] \arrow[d, "q"'] 
& \mathbb{R}^2 \arrow[d, "q"] \\
Q \arrow[r, "f"] 
& Q
\end{tikzcd}
\]

\begin{proof}
    For $x \in Q$, suppose $y_1$ and $y_2$ are two points in $q^{-1}(x)$. Then $y_2 = \rho(y_1)$ for some isometry $\rho$, where $\rho = \rho_1 \circ \rho_2 \circ \cdots \circ \rho_n$, and for $1 \leq i \leq n$ each $\rho_i$  is a $180^\circ$ rotation around the center of a special hexagon.  Let $\rho_{i}$ be one such  $180^\circ$ rotation, with center of rotation $z$ in the center of a special hexagon. Since $t \circ \rho_{i}  \circ t^{-1}$ is the conjugate of a $180^\circ$ rotation, it must also be a $180^\circ$ rotation. Since it fixes $t(z)$, it must have center of rotation $t(z)$. Since $t$ takes special hexagons to special hexagons, $t(z)$ is at the center of a special hexagon, so $t \circ \rho_{i} \circ t^{-1}$ is a $180^\circ$ rotation around the center of a special hexagon. Therefore, $t \circ \rho \circ t^{-1} = t \circ \rho_1 \circ t^{-1} \circ t\circ \rho_2 \circ t^{-1} \circ \cdots \circ t \circ \rho_n \circ t^{-1}$ is a product of $180^\circ$ rotations around centers of special hexagons, so $q \circ t \circ \rho \circ t^{-1} = q$. Therefore, $q(t(y_2)) = q(t \circ \rho(y_1)) = q(t \circ \rho \circ t^{-1}(t(y_1))) = q(t(y_1))$.  So $f(x)$ is well-defined. 

    Since $t$ is an isometry that takes special hexagons to special hexagons, it must take all hexagons to hexagons in the hexagonal tiling for the plane. Therefore, $f$ takes vertices and edges of $Q$ to vertices and edges of $Q$, respectively, and induces a graph homomorphism on $Q$. Since $t$ is an isometry, it is invertible, so $f$ is also invertible and this graph homomorphism must be an isomorphism.
\end{proof}

\begin{proposition}
For a $(3,6)$-fullerene, the following are equivalent:
\begin{enumerate}
\item The $(3,6)$-fullerene has 3-fold rotational symmetry.

\item All three of the equivalent signatures of the $(3,6)$-fullerene are the same triple of numbers, i.e. the $(3, 6)$-fullerene has coinciding signatures.
\item The $(3,6)$-fullerene has at least two equivalent signatures that are the same triple of numbers.

\item In the hexagonal tiling of the plane that covers the $(3,6)$-fullerene, any rotation by $60^\circ$ around the center of a special hexagon takes special hexagons to special hexagons.

\item In the hexagonal tiling of the plane that covers the $(3,6)$-fullerene, there is some rotation by $120^\circ$ that takes special hexagons to special hexagons.

\end{enumerate}
\label{prop-3foldSymmetry}
\end{proposition} 

\begin{proof}
    (1) $\implies$ (2): Suppose the $(3,6)$-fullerene has 3-fold rotational symmetry. That is, it is equivalent to an embedding on the sphere in which a 3-fold rotation $r$ of the sphere induces a graph isomorphism. 
    Since this rotation $r$ takes triangles to triangles, and there are four triangles, it must take at least one triangle to itself. The rotation $r$ cannot fix the edges of this triangle, since it would then be the identity, so it must permute them. 
    Consequently, the rotation permutes the three spines exiting that triangle. Therefore, the length of the spines is the same in all three signatures of the $(3,6)$-fullerene. Since $r$ takes hexagons to hexagons and triangles to triangles, the number of belts and the offsets are also the same in all three signatures.
    
    $(2) \implies (3)$ is clear. 
    
    $(3) \implies (4)$: Suppose the $(3,6)$-fullerene has two equivalent signatures that are the same triples of numbers (e.g. $(s_1, b_1, f_1) = (s_2, b_2, f_2)$). Then the hexagonal tiling that covers the $(3,6)$-fullerene has the same configuration of special hexagons, whether using spines in either of two directions (e.g.  vertical direction and SW to NE direction). Therefore, a rotation of the plane by $60^\circ$ clockwise or counterclockwise through the center of a special hexagon will take special hexagons to special hexagons. 

    $(4) \implies (5)$: Assume that any rotation by $60^\circ$ around the center of a special hexagon takes special hexagons to special hexagons. Choose one such rotation $r$. Then $r^2$ is a rotation by $120^\circ$ that takes special hexagons to special hexagons. 
    
    $(5) \implies (1)$: Suppose that there is rotation $r$ by $120^\circ$ that takes special hexagons to special hexagons. Let $Q$ be the quotient space of the plane $\mathbb{R}^2$ under the group of isometries generated by $180^\circ$ rotations around centers of special hexagons. Let $q$ be the quotient map $q:\mathbb{R}^2 \to Q$. By Lemma~\ref{lem-welldefined}, the rotation $r$ induces a well-defined map $f:Q \to Q$ defined by $f(x) = q(r(y))$, where $y$ is any point in $q^{-1}(x)$, and this map $f$ induces a graph isomorphism on the vertices and edges of $Q$. Since $r$ has order 3 and $f$ is not the identity, $f$ has order 3 also.  Since $Q$ is homeomorphic to a sphere \cite{green2024polyhedra}, there is an embedding of the $(3,6)$-fullerene on the sphere and an order 3 homeomorphism of the sphere that induces a graph isomorphism on this embedding. Since every homeomorphism of the sphere is conjugate by a homeomorphism to an isometry \cite{Kerekjaro1919}, it follows that the $(3, 6)$-fullerene is equivalent to an embedding on the sphere for which an order 3 isometry of the sphere induces a graph isomorphism. This order 3 isometry must be a rotation.

\end{proof}

\vspace{1 cm}

\section{Coinciding signatures and the equation \texorpdfstring{$x^2 + x + 1 \equiv 0 \pmod n$}{x² + x + 1 ≡ 0 (mod n)}}

\label{sec-equation}

This section gives an algebraic condition for coinciding signatures; that is, an algebraic condition for which $(3,6)$-fullerenes have 3-fold rotational symmetry.

\begin{lemma}\label{lem-ifDupThenBPlus1IsFactor} If a $(3,6)$-fullerene with signature $(s, b, f)$ has coinciding signatures, then $b + 1$ divides both $s + 1$ and $f$. 
\end{lemma}

\begin{proof} 
By Equation~\ref{eqn-s2} in Section~\ref{sec-background}, we know that $s + 1 = j_2(b+1)$, so  $b+1$ must divide $s+1$. From Equation~\ref{eqn-f2}, $f \equiv -p_2(b+1) - (b+1) \equiv -(p_2 + 1)(b + 1) \pmod {s+1}$. Therefore, $s+1$ divides $f + (p_2 + 1)(b+1)$. Since $b+1$ divides $s+1$ and $s+1$ divides $f + (p_2 + 1)(b+1)$, $b+1$ divides $f + (b+1)(p_2 + 1)$. Therefore, $b + 1$ divides $f$.

\end{proof}

\begin{proposition} \label{thm-dupSigsFmla} A $(3,6)$-fullerene has coinciding signatures if and only if it has a signature of the form $(cm - 1, m-1, gm)$, where $1 \leq c$,  $1 \leq m$, $0 \leq g < c$, and $g^2 + g + 1 \equiv 0 \pmod {c}$.
\end{proposition}

\begin{proof}
Suppose that a $(3,6)$-fullerene has coinciding signatures $(s, b, f)$. By Lemma~\ref{lem-ifDupThenBPlus1IsFactor}, $s+1$ and $f$ are both multiples of $b+1$. Let $m = b+1$. Then $b = m-1$, $s+1 = cm$ for some integer $c$, and $f = gm$ for some integer $g$. So the signature must be of the form $(cm - 1, m - 1, gm)$. Note that $m \geq 1$ since $b \geq 0$ and $c \geq 1$ since $s \geq 0$. We can assume $ 0 \leq g < c$, since we can choose $f$, defined mod $s+1$, such that $0 \leq f \leq s$. We must still show that $g^2 + g + 1 \equiv 0 \pmod c$.

If $c = 1$, then $g = 0$, so  $f = 0$, and the signature is of the form $(m - 1, m - 1, 0)$. Note that the equation $g^2 + g + 1 \equiv 0 \pmod c$ holds trivially since $c = 1$. 

Suppose $c \neq 1$. 
Consider the number $p_2$ in Equation~\ref{eqn-f2}: $p_2$ is the smallest positive integer such that $p_2 f_1 \equiv (b_2 + 1) \pmod {s_1+ 1}$, that is, the smallest positive integer such that $p_2 gm \equiv m \pmod {cm}$. This is also the smallest positive integer such that $p_2 g \equiv 1 \pmod c$. 

From Equation~\ref{eqn-f2}, we have that $gm \equiv f_2 \equiv  - p_2(b_1 + 1) - (b_2+1) \pmod {s_2 + 1}$, so $gm \equiv  - p_2(m)- (m)  \equiv ( - p_2 - 1)m \pmod {cm}$, which means that $g \equiv ( - p_2 - 1) \pmod {c}$. Multiplying both sides by $g$ we get that $g^2 \equiv g(-p_2 - 1) \equiv -g p_2 - g \equiv -1 - g \pmod {c}$. Here, we use the fact that $p_2 g \equiv 1 \pmod{c}$. Therefore, $g^2 + g  + 1 = 0 \pmod{c}$.

Conversely, suppose that a $(3,6)$-fullerene has a signature $(s_1, b_1, f_1)$ of the form \mbox{$(cm - 1, m - 1, gm)$} with $1 \leq c$, $0 \leq g < c $, and $g^2 + g + 1 \equiv 0 \pmod c$.

Since $g^2 + g + 1 \equiv 0 \pmod c$, $g$ and $c$ are relatively prime. The order of $gm$ in $\mathbb{Z}_{cm}$ is therefore $c$. So in  Equation~\ref{eqn-s2}, $j_2  = c$ and $s_2 = j_2(b_1 + 1) - 1  = cm - 1$. Therefore, $s_2 = s_1$.

By Equation~\ref{eqn-b2}, $b_2 =  \dfrac{h - 2s_2}{2s_2 + 2} = \dfrac{h - 2s_1}{2s_1 + 2} = b_1$,  so $b_2$ and $b_1$ are also the same. 

To verify that $f_2$ is the same as $f_1$, using Equation~\ref{eqn-f2}, we need to find the smallest positive integer $p_2$ such that $p_2 f_1 \equiv (b_2 + 1) \pmod {s_1 + 1}$, that is, such that $p_2 gm \equiv m \pmod {cm}$. Equivalently, we need the smallest positive integer $p_2$ such that $p_2 g \equiv 1 \pmod {c}$. 

If $c = 1$, then since $0 \leq g < c$, $g = 0$ and the first signature is $(m-1, m - 1, 0)$.  Also $p_2 = 1$, since $p_2$ is the smallest positive integer such that $p_2 \cdot 0 \equiv 1 \pmod {1}$. By Equation~\ref{eqn-f2}, $f_2 \equiv  - p_2(b_1 + 1) - (b_2 + 1) \pmod{s_2 + 1}$, so $f_2 =  -1(m) -(m)  \equiv 0 \pmod {1\cdot m}$, and the signature $(s_2, b_2, f_2)$ is also $(m-1, m-1, 0)$, coinciding with the original signature $(s_1, b_1, f_1)$. 

Suppose that $c > 1$. We will show that $p_2 = c - g - 1$.    Note that  $(c - g - 1 ) g \equiv cg - g^2 - g \equiv -(g^2 + g) \pmod c$, and $g^2 + g + 1 \equiv 0 \pmod c$, so $(c - g - 1 ) g  \equiv 1 \pmod c$. Since $g$ and $c$ are relatively prime, there is a unique solution $x$ to the equation $x g \equiv 1 \pmod{c}$ with $0 \leq x < c$. Since $c - g - 1$ solves this equation and $0 \leq c - g - 1 < c$, this unique solution is $c - g - 1$. Also $c - g - 1 \neq 0$ since $(c - g- 1)g \equiv 1 \pmod{c}$ and $c \neq 1$. Since  $p_2$ is the smallest positive integer such that $p_2 g \equiv 1 \pmod{c})$, $p_2$ must equal  $c - g - 1$. Therefore, according to Equation~\ref{eqn-f2}, $f_2 \equiv - p_2(b_1 + 1) - (b_2+1)  \equiv  - (c - g - 1)(m ) - (m) \equiv -cm +  gm \equiv gm \pmod{s_2 + 1}$, since $s_2 + 1 = cm$. So $f_2 = f_1$. 

Therefore, $(s_2, b_2, f_2) = (s_1, b_1, f_1)$. Since two signatures coincide, by Proposition~\ref{prop-3foldSymmetry} all three must coincide.

\end{proof}

Proposition \ref{thm-dupSigsFmla} shows that the equation $x^2 + x + 1 \equiv 0 \pmod n$ is key to detecting coinciding signatures. The number of solutions to this equation is given in Proposition 3.2 of \cite{mohar1987enumeration}:

\begin{fact} \cite{mohar1987enumeration} For $n \geq 1$, define 
$\Omega(n) := \#\{x \in \mathbb{Z}_n : x^2 + x + 1 \equiv 0 \pmod{n}\}$.
Let $n = \displaystyle 3^{y}\prod_{i=1}^z p_i^{k_i}$
 be the prime decomposition of $n$, where the $p_i$ are distinct primes other than 3, $y \geq 0$, $z \geq 0$, and $k_i \geq 1$ for $1 \leq i \leq z$. If $y \in \{0, 1\}$ and  $p_i \equiv 1 \pmod 3$ for $1 \leq i \leq z$, then $\Omega(n) = 2^z$. In any other case, $\Omega(n) = 0$.

\label{fact-product} 
\end{fact}

Note that the value of $\Omega(n)$ depends on the exponent of the prime $3$ in the prime factorization, but otherwise, only on which primes appear in the prime factorization and not their exponents. For example, $\Omega(63) = 0$ since $63 = 3^2 \cdot 7$, and $3$ is raised to a power greater than 1. Also, $\Omega(105) = 0$, since $105 = 3 \cdot 5 \cdot 7$ and $5 \not\equiv 1 \pmod 3$. However, $\Omega(637) = \Omega(273) = 2^2$, since $637 = 7^2 \cdot 13$ and $273 = 3 \cdot 7 \cdot 13$,  so both numbers contain exactly two primes not equal to 3 in their prime decomposition, the primes $7$ and $13$, and $7 \equiv 1 \pmod 3$ and $13 \equiv 1 \pmod 3$, .

\section{The number of (3,6)-fullerenes with a given number of vertices}

\label{sec-(3,6)-fullereneCounts}
In the counting formulas of this section and subsequent sections, we will say ``the number of $(3, 6)$-fullerenes'' as shorthand for ``the number of equivalence classes of $(3, 6)$-fullerenes''. Recall that two $(3,6)$-fullerenes are equivalent if they are not only isomorphic as graphs but there is also an orientation-preserving homeomorphism of the plane that takes one graph to the other. In other words, left and right handed $(3,6)$-fullerenes are considered distinct even though they are isomorphic as graphs.

Let $\sigma(V)$ represent the number of signatures that represent a $(3,6)$-fullerene with $V$ vertices. If $V$ is not a multiple of 4, then $\sigma(V) = 0$. Otherwise, for each signature $(s, b, f)$, $s+1$ is a factor of $V$, since $V = 4(s+1)(b+1)$  by Equation~\ref{eqn-V}.  Once this factor is chosen, $b$ is determined by this same equation. For each pair $(s, b)$, there are $s+1$ possible values of $f$, since $f$ is defined mod $s+1$. Therefore, for each factor $s+1$ of $\dfrac{V}{4}$, there are $s+1$ possible signatures, so $\sigma(V)$ is equal to the sum of the factors of $\dfrac{V}{4}$. This sum of factors is given by:

\begin{equation}
\displaystyle \sigma(V) = \prod_{i = 1}^z \dfrac{p_i^{k_i + 1} - 1}{p_i - 1}
\label{eqn-sigma}
\end{equation}
where  $\displaystyle \prod_{i = 1}^z p_i^{k_i}$ is the prime factorization of $\dfrac{V}{4}$, as stated in Lemma 2 of \cite{green2024polyhedra}.

Let $\delta(V)$ be the number of $(3,6)$-fullerenes with 3-fold rotational symmetry with $V$ vertices.   When $V$ is not a multiple of 4, $\delta(V) = 0$. When $V$ is a multiple of 4, 
the following theorem computes $\delta(V)$ in terms of the prime factorization of $\dfrac{V}{4}$.

\begin{theorem} Suppose $V$ is a multiple of 4 and $\frac{V}{4}$ has prime factorization $\displaystyle 3^y \prod_{i = 1}^z p_i^{k_i} \prod_{j = 1}^w q_j^{\ell_j}$  with distinct primes $p_i \equiv 1 \pmod 3$ and distinct primes $q_j \equiv 2 \pmod 3 $, where $z, w,$ and the exponent $y$ can be any integer $\geq 0$, and the other exponents are all $\geq 1$. Then the number $\delta(V)$ of $(3,6)$-fullerenes with $V$ vertices and 3-fold rotational symmetry  is 0 if $\ell_j$ is odd for any $j$ with $1 \leq j \leq w$. If all $\ell_j$ are even, then  $\delta(V) = \displaystyle \prod_{i = 1}^z (k_i + 1)$.
\label{thm-delta}
\end{theorem}

\begin{proof}

By Proposition~\ref{prop-3foldSymmetry}, the number of $(3,6)$-fullerenes with $V$ vertices with 3-fold rotational symmetry is the same as the number of $(3,6)$-fullerenes with $V$ vertices with coinciding signatures. By Theorem~\ref{thm-dupSigsFmla}, the number of $(3,6)$-fullerenes with $V$ vertices with coinciding signatures is the number of triples $(cm-1, m-1, gm)$ where $1 \leq c$, $1 \leq m$, $0 \leq g < c$, with $g^2 + g + 1 \equiv 0 \pmod c$ and $4cm^2 = V$. The last equation comes from the fact that $V = 4(s+1)(b+1)$ (Equation~\ref{eqn-V}), where $s = cm  -1$ and $b = m-1$.

Therefore, the number of $(3,6)$-fullerenes with 3-fold rotational symmetry with $V$ vertices can be found by counting the solutions to $g^2 + g + 1 \equiv 0 \pmod c$ for all possible ways of factoring $\dfrac{V}{4}$ as $cm^2$. 

By Fact~\ref{fact-product}, $g^2 + g + 1 \equiv 0 \pmod c$ will have 0 solutions if $c$ has any prime factors that are congruent to $2 \pmod 3$ or has a prime factor of 3 raised to power higher than 1. So the only options for $c$ and $m$ in the factorization $\dfrac{V}{4} = cm^2$ that generate solutions are those options where all prime factors congruent to $2 \pmod 3$ are put into $m^2$ and only prime factors congruent to $1 \pmod 3$ and one or zero copies of $3$ are put into $c$. Therefore, it is only possible to get solutions if all exponents $\ell_j$ for $1 \leq j \leq w$ are even. 

In this case, in order to get solutions, $m^2$ must include $q_1^{\ell_1} q_2^{\ell_2} \cdots q_w^{\ell_w}$ as a factor. In addition, $m^2$ must include the highest possible even power of $3$ in the prime factorization $\dfrac{V}{4}$ so that $c$ contains only one factor of 3 if $y$ is odd, or no factors of 3 if $y$ is even. For every prime $p_i$ with exponent $k_i$ in the prime factorization of $\dfrac{V}{4}$, $m^2$ can contain any even power of $p_i$, where the power is less than or equal to $k_i$, with the remaining factors of $p_i$ going into $c$. Therefore, $c$ can be factored as $3^{r} \cdot p_1^{t_1} p_2^{t_2} \cdots p_z^{t_z}$ where $r$ is either 0 or 1, and for each $i$ with $ 1 \leq i \leq z$, we have $0 \leq t_i \leq k_i$  and  $t_i$ has the same parity as $k_i$. 

Recall that for $c = 3^{r} \cdot p_1^{t_1} p_2^{t_2} \cdots  p_z^{t_z}$, if $d$ is the number of non-zero exponents among $\{t_1, t_2, \cdots , t_z\}$, then there are exactly $2^d$ solutions to $g^2 + g + 1 \equiv 0 \pmod c$,  by Fact~\ref{fact-product}. 

For each exponent $k_i$, if $k_i$ is even, form the expression $(1 + 2 + 2 + 2 + \cdots + 2)$, with one 2 for each positive even number $\leq k_i$. If $k_i$ is odd, form the expression $(2 + 2 + \cdots + 2)$, with one 2 for each positive odd number $\leq k_i$. If we multiply all these expressions together, with one factor for each $k_i$, and distribute, then each term will correspond to exactly one value of $c$ that generates solutions, and the number of solutions for that value of $c$ will be exactly the value of the term. Therefore, the product will give exactly the number of $(3,6)$-fullerenes with $V$ vertices with 3-fold rotational symmetry. 

If $k_i$ is even, there are $\dfrac{k_i}{2}$ positive even numbers less than or equal to $k_i$, so the sum $(1 + 2 + 2 + \cdots + 2)$ is equal to $\dfrac{k_i}{2} \cdot 2 + 1 = k_i + 1$.  If $k_i$ is odd, there are $\dfrac{k_i + 1}{2}$ positive odd numbers less than or equal to $k_i$, so the expression $(2 + 2 +  \cdots + 2)$ will equal $\dfrac{k_i + 1}{2} \cdot 2 = k_i + 1$ also.  

Therefore, the number $\delta(V)$ of  $(3,6)$-fullerenes with $V$ vertices and 3-fold rotational symmetry is given by $\displaystyle \prod_{i = 1}^z (k_i + 1)$. 
\end{proof}
An alternative proof of this formula, using ideals in the Eisenstein lattice, can be found in Theorem 4.3 of J. Huijben's thesis \cite{huijben2016tetrahedral}. 

We can now give an exact count of the number of equivalence classes of $(3,6)$-fullerenes with $V$ vertices, which we will call $\alpha(V)$.  If $V$ is not a multiple of 4, then $\alpha(V) = 0$, since the number of vertices of a $(3,6)$-fullerene is always a multiple of 4 by Equation~\ref{eqn-V}. Otherwise, $\alpha(V)$ is given by the following theorem.

\begin{theorem} Let $\alpha(V)$ be the number of $(3, 6)$-fullerenes with $V$ vertices, $\delta(V)$ be the number of $(3, 6)$-fullerenes with $V$ vertices with 3-fold rotational symmetry, and $\sigma(V)$ be the number of signatures of $(3, 6)$-fullerenes with $V$ vertices. Then $$\alpha(V) = \dfrac{1}{3} \sigma(V)+ \dfrac{2}{3} \delta(V).$$
Suppose $V$ is a multiple of 4 and $\frac{V}{4}$  has prime factorization $\displaystyle 3^y \prod_{i = 1}^z p_i^{k_i} \prod_{j = 1}^w q_j^{\ell_j}$ with distinct primes $p_i \equiv 1 \pmod 3$ and distinct primes $q_j \equiv 2 \pmod 3 $, where $z, w$, and the exponent $y$ can be any integer $\geq 0$, and the other exponents are all $\geq 1$.

If  $\ell_j$ is odd for any $j$ with $1 \leq j \leq w$, then 

$$\displaystyle \alpha(V) =  \dfrac{1}{3} 
\left( \dfrac{3^{y + 1} - 1}{3 - 1}\prod_{i = 1}^z \dfrac{p_i^{k_i + 1} - 1}{p_i - 1} \prod_{j=1}^w \dfrac{q_j^{\ell_j + 1} -  1}{q_j - 1}\right)$$ 

If all  $\ell_j$ are even, then 

$$\displaystyle \alpha(V) = \dfrac{1}{3} 
\left( \dfrac{3^{y + 1} - 1}{3 - 1}\prod_{i = 1}^z \dfrac{p_i^{k_i + 1} - 1}{p_i - 1} \prod_{j =1}^w \dfrac{q_j^{\ell_j + 1} -  1}{q_j - 1}\right) + \frac{2}{3} \left( \prod_{i = 1}^z (k_i + 1) \right)$$.
\label{thm-$(3,6)$-fullereneCount}
\end{theorem}

\begin{proof}
The number of $(3,6)$-fullerenes with $V$ vertices is given by $\dfrac{1}{3} (\sigma(V) - \delta(V)) + \delta(V)$, since there are three distinct equivalent signatures that describe each $(3,6)$-fullerene without 3-fold rotational symmetry and only one signature that describes each $(3,6)$-fullerene with 3-fold rotational symmetry. This expression can be rewritten as $\dfrac{1}{3} \sigma(V)+ \dfrac{2}{3} \delta(V)$. The other formulas in the theorem follow from the formula for $\sigma(V)$ from Equation~\ref{eqn-sigma} and the formula for $\delta(V)$ given in Theorem~\ref{thm-delta}.
\end{proof}

See Table~\ref{table:bigCounts} for the number of $(3,6)$-fullerenes for values of $V$ up to 480.

\section{Mirror symmetry}
\label{sec-mirror}

We will say that a $(3,6)$-fullerene has \emph{mirror symmetry} if it is equivalent to an embedding on the sphere for which a reflection of the sphere induces an automorphism of the graph.

\begin{figure}[ht]

\begin{center}
\includegraphics[height = 6 cm]{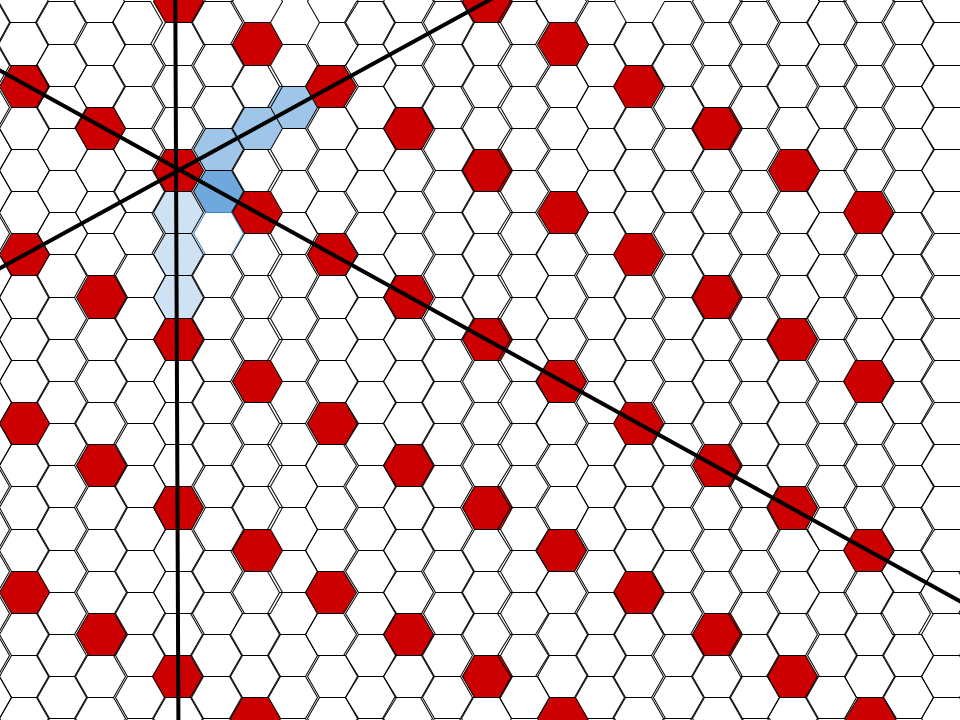}
\end{center}
    \caption{Three lines that bisect spines. Special hexagons are shaded dark red (or dark gray if printed in grayscale). Only the line in the NW to SE direction is a mirror line.}
   \label{fig:bisectSpines}
\end{figure}

We will say that a line in the hexagonal covering of a $(3,6)$-fullerene \emph{bisects spines} if it passes through the center of a spine column in either the vertical direction, the SW to NE direction, or the NW to SE direction. See Figure~\ref{fig:bisectSpines}.

Recall that equivalent signatures are signatures that are either equal or related by equations~\ref{eqn-s2} - \ref{eqn-f3} in Section~\ref{sec-background}. By Theorem~4 of \cite{green2024polyhedra}, two $(3, 6)$-fullerenes have equivalent signatures if and only if they are equivalent $(3, 6)$-fullerenes.

\pagebreak
\begin{proposition}
For a $(3,6)$-fullerene, the following are equivalent
\begin{enumerate}
    \item The $(3,6)$-fullerene  has mirror symmetry.
\item For every signature of the $(3,6)$-fullerene $(s, b, f)$, the signature $(s, b, s-b-f )$ is equivalent to $(s, b, f)$.
    \item For some signature of the $(3,6)$-fullerene $(s, b, f)$, the signature $(s, b, s-b-f )$ is equivalent to $(s, b, f)$.
    \item In the hexagonal tiling of the plane that covers the $(3,6)$-fullerene, there is a reflection through a mirror line that bisects spines, that takes special hexagon to special hexagons.
    \item In the hexagonal tiling of the plane that covers the $(3,6)$-fullerene, there is a reflection through a mirror line that takes special hexagon to special hexagons.

\end{enumerate} 
\label{prop-mirrorSymmetry}
\end{proposition}

\begin{proof}
$(1) \implies (2)$: Suppose that the $(3,6)$-fullerene has mirror symmetry. That is, it is equivalent to an embedding on the sphere for which a reflection of the sphere induces a graph automorphism. Let $(s, b, f)$ be one of the $(3,6)$-fullerene's three signatures.
By Proposition 3 of \cite{green2024polyhedra}, the mirror image of a $(3,6)$-fullerene with signature $(s, b, f)$ is a $(3,6)$-fullerene with signature $(s, b, s-b-f )$. So the $(3,6)$-fullerene with signature $(s, b, s-b-f )$ also has signature $(s, b, f)$. By Theorem 4 of \cite{green2024polyhedra}, the signature $(s, b, s- b - f )$ is equivalent to $(s, b, f)$.

Clearly $(2) \implies (3)$. 

$(3) \implies (4)$: Suppose $(3)$ holds. Consider the hexagonal tiling that covers the $(3,6)$-fullerene with signature $(s, b, f)$ and orient it so that the spines corresponding to this signature are vertical. Let $m$ be a reflection of the plane through a vertical mirror line that bisects spines. The image of the hexagonal tiling under this reflection will have signature $(s, b, s-b-f )$, by Proposition 3 of \cite{green2024polyhedra}.

By assumption, the  signature $(s, b, s-b - f )$ is equivalent to the signature $(s, b, f)$. So either $(s, b, s-b-f )$ is equal to $(s, b, f)$, or else $(s, b, s-b-f )$ is an alternative equivalent signature for $(s, b, f)$. If the two signatures are equal, then the reflection $m$ through a vertical mirror line that bisects spines creates a hexagonal tiling is identical to the original, so (4) is proved. 

If, instead, $(s, b, s-b-f )$ is an alternative signature for $(s, b, f)$, then there is a rotation $r_{120}$ by $120^\circ$ clockwise or counterclockwise, with center of rotation in a special hexagon, that takes the mirror image hexagonal tiling to the original hexagonal tiling. So $r_{120} \circ m$ takes special hexagons to special hexagons. But $r_{120} \circ m$ is itself a reflection through a mirror at a $60^\circ$ angle to the original that goes through special hexagons  either in the SW to NE direction or in the NW to SE direction. So there is still a reflection through a mirror line that bisects spines that takes special hexagons to special hexagons. 

Clearly $(4) \implies (5)$.

$(5) \implies (1)$: Suppose that there is a reflection through a mirror line that takes special hexagons to special hexagons, in the hexagonal tiling that covers the $(3,6)$-fullerene. Call this reflection $m$. Let $Q$ be the quotient space of the plane $\mathbb{R}^2$ under the group of isometries generated by $180^\circ$ rotations around centers of special hexagons. Let $q$ be the quotient map $q:\mathbb{R}^2 \to Q$. By Lemma~\ref{lem-welldefined}, the map $m$ induces a well-defined map $f:Q \to Q$ defined by $f(x) = q \circ m(y)$, where $y$ is any point in $q^{-1}(x)$,  and the map $f$ induces a graph automorphism on the edges and vertices of $Q$. Since $m$ is orientation reversing and has a fixed point, $f$ is also orientation reversing with a fixed point. Since $Q$ is homeomorphic to a sphere, there is an embedding of the $(3,6)$-fullerene on a sphere such that an orientation reversing transformation with a fixed point is a graph automorphism. By \cite{Kerekjaro1919}, this transformation must be conjugate to a reflection of the sphere, so the $(3,6)$-fullerene is equivalent to an embedding on the sphere for which a reflection of the sphere induces a graph automorphism.

\end{proof}

By Proposition~\ref{prop-mirrorSymmetry}, a $(3,6)$-fullerene with signature $(s, b, f)$ has mirror symmetry if and only if the signature $(s, b, s-b-f )$ is either identical to $(s, b, f)$ or equivalent but not identical to $(s, b, f)$, where equivalence is given by the algebraic rules for alternative signatures in Equations~\ref{eqn-s2}-\ref{eqn-f3}. Recall that the third number in the signature is defined mod ${s+1}$, so $(s, b, s-b-f)$ is identical to $(s, b, f)$ means that $s-b-f \equiv f \pmod {s+1}$. We will call the signature $(s, b, s-b-f )$ the \emph{mirror signature} for the signature $(s, b, f)$. We will say that a signature $(s, b, f)$ is \emph{self-mirror} if its mirror signature $(s, b, s-b-f \pmod{s+1})$ is identical to $(s, b, f \pmod{s+1})$, not just equivalent. For example, $(4, 2, 1)$, $(14, 0, 11)$, and $(14, 0, 3)$ are all equivalent signatures for a $(3,6)$-fullerene with mirror symmetry. The signature $(4, 2, 1)$ is self-mirror, but the signatures $(14, 0, 11)$ and $(14, 0, 3)$ are not self-mirror. Instead, they are mirror signatures of each other.

\begin{lemma}
\label{lemma-mirrorTriples}
Suppose a $(3,6)$-fullerene has mirror symmetry and does not have 3-fold rotational symmetry. Then one of its three equivalent signatures is self-mirror, and the other two equivalent signatures are each other's mirror signatures.

\end{lemma}

\begin{proof}
If a $(3,6)$-fullerene has mirror symmetry and does not have 3-fold rotational symmetry, then by Proposition~\ref{prop-mirrorSymmetry}, the function $F$ that takes a signature $(s, b, f)$ to its mirror signature $(s, b, s -  b - f )$ must permute the $(3,6)$-fullerene's three signatures. Since $F^2$ is the identity,  $F$ must either fix all three equivalent signatures or it must swap two signatures and fix the third. 

Suppose, for contradiction, that a $(3,6)$-fullerene has mirror symmetry and does not have 3-fold rotational symmetry and that $F$ fixes all three equivalent signatures. 
In particular, $f_1 \equiv (s_1 - b_1 - f_1) \pmod {s_1 + 1}$. Let $j_2$ be the order of $f_1$ and $j_3$ be the order of $f_1 + b_1 + 1$ in $\mathbb{Z}_{s_1 + 1}$ as in equations~\ref{eqn-s2} and \ref{eqn-s3}. Since $f_1 = s_1 - b_1 - f_1 = -(f_1 + b_1 + 1) \pmod {s_1 + 1}$, the order of $f_1$ in $\mathbb{Z}_{s_1 + 1}$ is equal to the order of $-(f_1 + b_1 + 1)$, which is equal to the order of $(f_1 + b_1 + 1)$. Therefore, $j_2 = j_3$. By Equations \ref{eqn-s2} and \ref{eqn-s3}, $s_2$ and $s_3$ are equal. By Equations \ref{eqn-b2} and \ref{eqn-b3}, $b_2$ and $b_3$ are equal. 

The same argument, with $(s_2, b_2, f_2)$ in the role of $(s_1, b_1, f_1)$ shows that $s_3$ and $s_1$ are equal and $b_3$ and $b_1$ are equal; thus $s_1 = s_2 = s_3$ and $b_1 = b_2 = b_3$. 

Since each signature is its own mirror signature, we have that $f_i \equiv s_i - b_i - f_i \pmod {s_i + 1} $, or equivalently, $2f_i \equiv  - (b_i +1)  \pmod {s_i + 1}$. Since the $s_i$ and $b_i$ are all equal, all three $f_i$ satisfy the equation $2x \equiv -(b_1 + 1) \pmod{s_1 + 1} $. But this equation can have at most $2$ distinct solutions mod $(s_1 + 1)$, since the greatest common divisor of $2$ and $(s_1 + 1)$ is at most 2. Therefore, at least two of the $f_i$ must be equal, meaning that at least two of the signatures are the same. Therefore, by Proposition~\ref{prop-3foldSymmetry}, all three of the signatures are the same: we have coinciding signatures. So the $(3,6)$-fullerene has 3-fold rotational symmetry, a contradiction.

We can conclude that if the $(3,6)$-fullerene does not have 3-fold rotational symmetry, then the function $F$ that takes each signature to its mirror signature must swap two signatures and fix the third. 

\end{proof}

\begin{corollary}
Let $G$ be the map that takes a signature to the $(3,6)$-fullerene that it represents.
\label{cor-geomAndSigs}
\begin{enumerate}

\item $G$ gives a bijection from signatures that are self-mirror to $(3,6)$-fullerenes with mirror symmetry.

\item $G$ gives a bijection from signatures that are both self-mirror and coinciding to $(3,6)$-fullerenes that have both mirror symmetry and 3-fold rotational symmetry.

\end{enumerate}
\label{cor-symmBij}
\end{corollary}

\begin{proof}
(1): By Proposition~\ref{prop-mirrorSymmetry}, a self-mirror signature always represents a $(3,6)$-fullerene with mirror symmetry, so the image of $G$ lies within the specified set of $(3,6)$-fullerenes. 

Every $(3,6)$-fullerene that has mirror symmetry but does not have 3-fold rotational symmetry has a unique self-mirror signature that represents it by Lemma~\ref{lemma-mirrorTriples}. Every $(3,6)$-fullerene that has mirror symmetry and does have 3-fold rotational symmetry also has a unique self-mirror signature, by the following argument.  By Proposition~\ref{prop-3foldSymmetry}, the $(3,6)$-fullerene has coinciding signatures; that is, all three of its signatures are identical. By Proposition~\ref{prop-mirrorSymmetry}, the mirror signature of this coinciding signature is equivalent to itself. But the only way a signature can be equivalent to a coinciding signature is if it identical to the signature, so the signature must be self-mirror. 

So every $(3,6)$-fullerene with mirror symmetry has a unique self-mirror signature that represents it, whether or not it has 3-fold rotational symmetry. The map that takes a $(3,6)$-fullerene to its unique self-mirror signature is an inverse to $G$, and $G$ must be a bijection. 

(2): This follows from (1) and the fact that a $(3,6)$-fullerene has coinciding signatures if and only if it has 3-fold rotational symmetry (Proposition~\ref{prop-3foldSymmetry}).

\end{proof}
\section{Counts of (3,6)-fullerenes with mirror symmetry}
\label{sec-mirrorCounts}
Let $\mu(V)$ be the number of $(3,6)$-fullerenes with $V$ vertices with mirror symmetry. Let $\nu(V)$ be the number of $(3,6)$-fullerenes with $V$ vertices that have both mirror symmetry and 3-fold rotational symmetry. By Corollary~\ref{cor-symmBij}, $\mu(V)$ is also the number of signatures $(s, b, f)$  that are self-mirror that represent a $(3,6)$-fullerene with $V$ vertices, and $\nu(V)$ is also the number of coinciding signatures  that are self-mirror that represent a $(3,6)$-fullerene with $V$ vertices.

The propositions that follow give formulas for $\mu(V)$ and $\nu(V)$ when $V$ is a multiple of 4. 
If $V$ is not a multiple of 4, there are no $(3,6)$-fullerenes with $V$ vertices by Equation~\ref{eqn-V}, so all of these counts are 0. 

\begin{lemma}
\label{lem-pMustBe3}
    For a positive integer $n$, there exists an integer $x$, such that $x^2 + x + 1 \equiv 0 \pmod {n}$ and $2x + 1 \equiv 0 \pmod n$ if and only if $n = 3$ or $n = 1$.
\end{lemma}

\begin{proof}
If $n = 1$, then $x = 0$ satisfies both $x^2 + x + 1 \equiv 0 \pmod n$ and $2x + 1 \equiv 0 \pmod n$. If $n = 3$, then $x = 1$ satisfies both equations. 

Conversely, suppose that $x^2 + x + 1 \equiv 0 \pmod {n}$ and $2x + 1 \equiv 0 \pmod n$. Then $2x + 1 = rn$ for some integer $r$, so $x = \dfrac{rn - 1}{2}$, and $x^2 + x + 1 = \left( \dfrac{rn -1}{2} \right)^2 + \dfrac{rn -1}{2} + 1 = \dfrac{r^2 n^2 + 3}{4} $. Since this is an integer congruent to $0 \pmod {n}$, $4n$ must divide $r^2 n^2 + 3$. So $n$ must divide 3, so $n = 1$ or $n = 3$.  
\end{proof}

\begin{proposition}
Suppose  $V$ is a multiple of $4$ and $\frac{V}{4}$  has prime factorization $\displaystyle 3^y \prod_{i = 1}^z p_i^{k_i}$, where the $p_i$ are distinct primes other than $3$, with $z \geq 0$, $y \geq 0$, and $k_i \geq 1$ for $1 \leq i \leq z$ . Then $\nu(V)$, the number of $(3,6)$-fullerenes with $V$ vertices with mirror symmetry and 3-fold rotational symmetry, is 1 if all exponents $k_i$ are even for $1 \leq i \leq z$, and 0 otherwise.
 
 A $(3,6)$-fullerene has both mirror symmetry and 3-fold rotation symmetry if and only if its signature is either of the form $(m-1, m-1, 0)$  or else of the form \mbox{$(3m-1, m-1, m)$} for some $m \geq 1$.
 \label{prop-nu}
\end{proposition}

\begin{proof} 
By Corollary~\ref{cor-geomAndSigs}, the $(3,6)$-fullerenes wih mirror symmetry and 3-fold rotational symmetry are exactly the $(3,6)$-fullerenes with coinciding and self-mirror signatures. By Theorem~\ref{thm-dupSigsFmla}, any coinciding signature for a $(3,6)$-fullerene with $V$ vertices is of the form $(cm-1, m-1, gm)$ where $c \geq 1$, $m \geq 1$, $0 \leq g < c$, $g^2 + g + 1 = 0 \pmod c$, and $cm^2 = \dfrac{V}{4}$. 
The mirror signature for this signature is $(cm-1, m-1, (cm-1) - (m-1) - gm)$. This mirror signature is identical to the original if and only if $(cm-1) - (m-1) - gm \equiv gm \pmod {cm}$ i.e. $2gm + m \equiv 0 \pmod {cm}$. This happens if and only if $2g + 1 \equiv 0 \pmod c$. By Lemma~\ref{lem-pMustBe3}, $g$ satisfies both $g^2 + g + 1 \equiv 0 \pmod c$ and $2g + 1 \equiv 0 \pmod c$ if and only if $c = 1$ or $c = 3$. When $c = 1$, $g = 0$ is the only possibility, and the signature is of the form $(m-1, m-1, 0)$, and when $c = 3$, $g = 1$ is the only possibility and the signature is of the form $(3m-1, m-1, m)$. Since there is exactly one signature for $c = 1$ and exactly one signature for $c = 3$ and no signatures for any other values of $c$, $\nu(V)$ is equal to the number of ways of factoring $\dfrac{V}{4}$ into $cm^2$ so that $c = 1$ or $c = 3$.

If any of the exponents $k_i$ are odd, then it is not possible to factor $\dfrac{V}{4}$ as $c m^2$ with $c = 1$ or $3$, so $\nu(V) = 0$. If all exponents $k_i$ are even, then there is one way to factor $\dfrac{V}{4}$ as $c m^2$ with $c = 1$ if $y$ is even and one way to factor $\dfrac{V}{4}$ as $cm^2$ with $c = 3$ if $y$ is odd.

Therefore, $\nu(V)$ is 1 if all exponents $k_i$ are even, and 0 otherwise.

\end{proof}

See Figure~\ref{fig:3foldAndMirror} for examples of $(3,6)$-fullerenes with both 3-fold rotational symmetry and mirror symmetry, represented by the hexagonal tilings that cover them.

\begin{figure}[ht]
\centering
    \includegraphics[height = 4 cm]{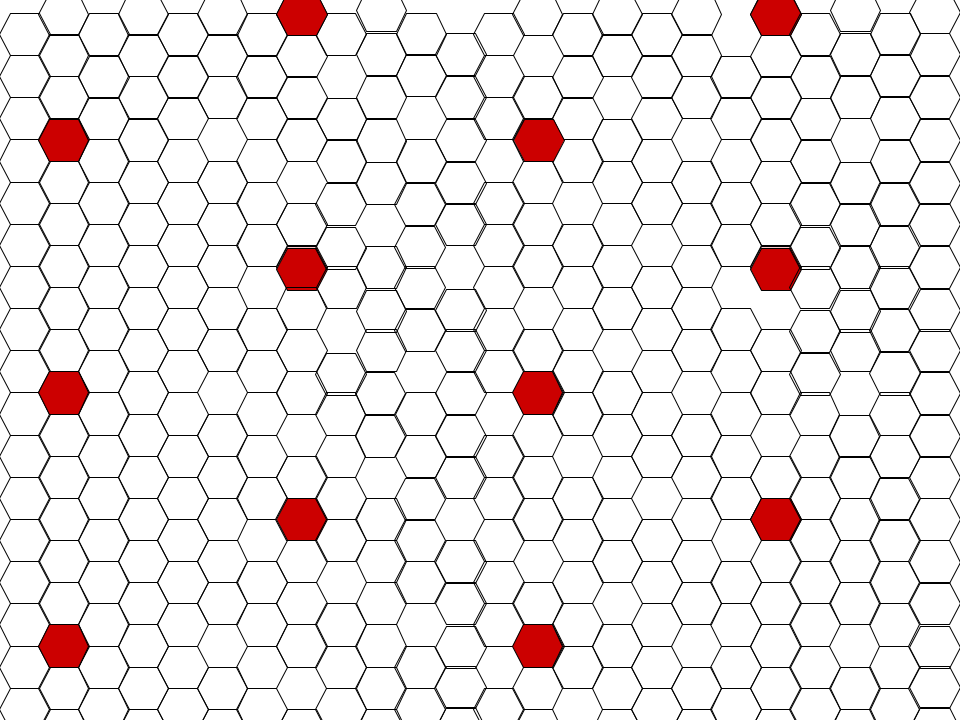}
    \hspace{1 cm}
    \includegraphics[height = 4 cm]{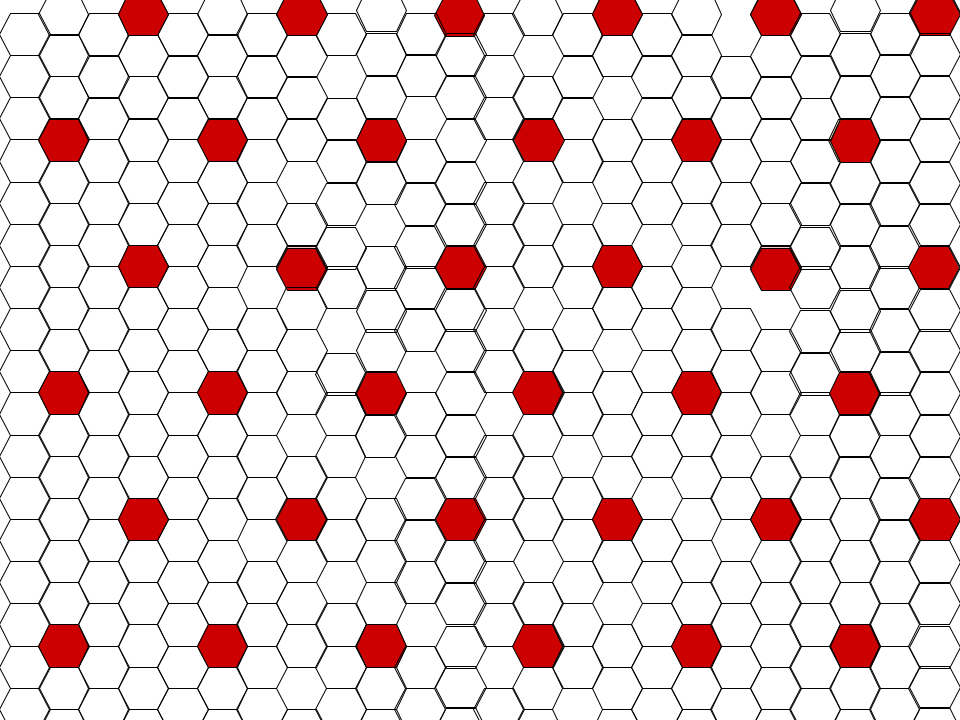}
\caption{Examples of $(3,6)$-fullerenes with both 3-fold rotational symmetry and mirror symmetry. The left figure shows the hexagonal cover of the $(3,6)$-fullerene with signature $(5, 5, 0)$, corresponding to $c = 1$. The right figure shows the hexagonal cover of the $(3,6)$-fullerene with signature $(5, 1, 2)$, corresponding to $c = 3$. Special hexagons are shaded. }
\label{fig:3foldAndMirror}
\end{figure}

Next, consider $\mu(V)$, the number of $(3,6)$-fullerenes with mirror symmetry. 

\begin{proposition}
Suppose $V$ is a multiple of $4$ and $\frac{V}{4}$ has prime decomposition $\displaystyle 2^x \prod_{i = 1}^z p_i^{n_i}$, where the $p_i$ are distinct primes other than 2, $z \geq 0$,  $x \geq 0$, and $n_i \geq 1$ for $1 \leq i \leq z$ . Then $\mu(V)$, the number of $(3,6)$-fullerenes with $V$ vertices with mirror symmetry, is given by  \\ $$\displaystyle \mu(V) = 
    (2x - 1) \cdot \prod_{i = 1}^z (n_i + 1)  \ \ \ \ \  \text{ if }x > 0 \text{ and}$$  $$\displaystyle \mu(V) =  \prod_{i = 1}^z (n_i + 1) \ \ \ \ \text{ if } x = 0.$$
    \label{prop-mu}
\end{proposition}

\begin{proof} By Corollary~\ref{cor-geomAndSigs}, the $(3,6)$-fullerenes with mirror symmetry are the $(3,6)$-fullerenes with a self-mirror signature, that is, $(3,6)$-fullerenes with a signature $(s, b, f)$ such that \mbox{$s-b-f \equiv f \pmod {s+1}$}, i.e. $2f \equiv -(b+1) \pmod {s+1}$. In general, the number of solutions for $x \in \mathbb{Z}_d$ in an equation $ax \equiv c \pmod d$ is equal to $\gcd(a, d)$ if $\gcd(a, d)$ divides $c$, and $0$ otherwise.   Therefore, the equation $2f \equiv -(b+1) \pmod {s+1}$ has two solutions for $f $ if $s+1$ and $b+1$ are both even, since then $\gcd(2, s+ 1) = 2$ and $2 \mid (b+1)$. The equation has one solution if $s+1$ is odd, since then $\gcd(2, s+1) = 1$ and $1 \mid (b+1)$. The equation has  no solutions if $s+1$ is even and $b+1$ is odd, since then $\gcd(2, s+1) = 2$ but $2 \nmid (b+1)$. 

Consider the prime decomposition $\displaystyle 2^x \prod_{i = 1}^z p_i^{n_i}$, where the $p_i$ are primes other than 2. 

Case 1: $x = 0$. Then all ways of factoring $\dfrac{V}{4}$ as $(s+1)(b+1)$ have an odd value for $s+1$ and therefore give one solution for $f$. The number of ways of factoring $\dfrac{V}{4}$ into two factors is $\displaystyle \prod_{i = 1}^z (n_i +1)$, so there are $\displaystyle \prod_{i = 1}^z (n_i +1)$ signatures that are their own mirror signature. 

Case 2: $x > 0$, i.e. there are factors of 2. Then there are $x - 1$ ways to divide up the 2's so that both $s+1$ and $b+1$ are even, 1 way so that $s+1$ is odd and $b+1$ is even, and 1 way so that $s+1$ is even and $b+1$ is odd. The other factors can be divided up in any way, so the total number of solutions is 
$$\displaystyle 2(x - 1) \cdot \prod_{i = 1}^z (n_i + 1)   + 1 \cdot \prod_{i = 1}^z (n_i + 1)   + 0 $$ This simplifies to 
$\displaystyle (2x - 1)\prod_{i = 1}^z (n_i + 1).  $
\end{proof}
A recursive version of this formula can be found in Theorem 4.4 of J. Huijben's thesis \cite{huijben2016tetrahedral}.

\section{The number of graph isomorphism classes of (3,6)-fullerenes }
\label{sec-(3,6)-fullereneGraphs}

Suppose we have a 3-connected $(3,6)$-fullerene $T_1$ with a fixed embedding in the plane. Let $T_2$ be any other $(3, 6)$-fullerene, with a fixed embedding in the plane, that is isomorphic to $T_1$ as a graph. By a theorem of Whitney \cite{whitneyCongruentGraphs1932}, there is a homeomorphism of the plane whose restriction to $T_1$ is a graph isomorphism onto $T_2$. This homeomorphism could be either orientation-preserving or orientation-reversing. If there exist both orientation-preserving and orientation-reversing homeomorphisms from the embedding of $T_1$ to the embedding of $T_2$, then the same is true for any $(3, 6)$-fullerene that is isomorphic as a graph to $T_1$, by composing homeomorphisms. So there is only one $(3,6)$-fullerene equivalence class within that graph isomorphism class, represented by both $T_1$ and its mirror image. In that case, we say that the $(3,6)$-fullerene is \emph{achiral}. If, instead, there exists an orientation-preserving homeomorphism from the embedding of $T_1$ to the embedding of $T_2$, but no orientation-reversing homeomorphism, or vice versa, then the same is true for any $(3,6)$-fullerene that is isomorphic as a graph to $T_1$, and there are exactly two distinct $(3, 6)$-fullerene equivalence classes within the same graph isomorphism class, represented by $T_1$ and its mirror image. In that case, we say that the $(3,6)$-fullerene $T_1$ is \emph{chiral}.

If the $(3,6)$-fullerene $T_1$ is not 3-connected, then it is the 2-connected godseye described in Section~\ref{sec-intro}. If $T_2$ is another $(3,6)$ embedded in the plane, whose graph is isomorphic to that of $T_1$, then by Whitney's 2-isomorphism theorem \cite{whitneyClassificationOfGraphs1933}, $T_1$ can be transformed by a series of flips along pairs of separating vertices to create an embedding $T_3$ that is equivalent to $T_2$. Since these flips do not change the face structure of $T_1$, $T_1$ is already equivalent to $T_2$. Since $T_1$ is also equivalent to its mirror image, there is one $(3,6)$-fullerene equivalence class for the graph isomorphism class of a godseye.

Let $\gamma(V)$ represent the number of graph isomorphism classes for $(3, 6)$-fullerenes, that is, the number of equivalence classes for $(3, 6)$-fullerenes where the equivalence relation is being isomorphic as graphs. To count $\gamma(V)$, we can use the count of equivalence classes of $(3, 6)$-fullerenes from Section~\ref{sec-(3,6)-fullereneCounts} and  collapse two equivalence classes corresponding to the mirror images  of chiral $(3, 6)$-fullerenes into a single graph isomorphism class.

\begin{lemma}
The number of graph isomorphism classes for $(3,6)$-fullerenes with $V$ vertices is 

 $\gamma(V) = \dfrac{1}{6} \sigma(V) + \dfrac{1}{3} \delta(V) +  \dfrac{1}{2}\mu(V)$.
\label{lem-$(3,6)$-fullereneGraphCount}
\end{lemma}

\begin{proof}

By Theorem~\ref{thm-$(3,6)$-fullereneCount}, the total number of $(3,6)$-fullerenes with $V$ vertices is $\alpha(V) = \dfrac{1}{3}\sigma(V) + \dfrac{2}{3}\delta(V)$.  The number of $(3,6)$-fullerenes with mirror symmetry is $\mu(V)$. Therefore, the number of $(3,6)$-fullerenes without mirror symmetry is $\dfrac{1}{3}\sigma(V) + \dfrac{2}{3}\delta(V)  - \mu(V) $.  For $(3,6)$-fullerenes with mirror symmetry, each $(3,6)$-fullerene corresponds to one graph isomorphism class, while for $(3,6)$-fullerenes without mirror symmetry, two $(3,6)$-fullerenes correspond to the same graph isomorphism class. Therefore, the number of $(3,6)$-fullerene graph isomorphism classes is  $ \gamma(V) = \mu(V) + \dfrac{1}{2}\left(\dfrac{1}{3}\sigma(V) + \dfrac{2}{3}\delta(V) - \mu(V)  \right)$. This simplifies to  \\
$\gamma(V) = \dfrac{1}{6} \sigma(V) + \dfrac{1}{3} \delta(V) +  \dfrac{1}{2}\mu(V)$.

\end{proof}
This formula is also proved using Burnside's lemma as Theorem 4.2 in J. Huijben's thesis \cite{huijben2016tetrahedral}. 

A precise, albeit complicated, formula can be obtained by combining Lemma~\ref{lem-$(3,6)$-fullereneGraphCount} with the  formulas for $\sigma(V)$, $\delta(V)$,  and $\mu(V)$.

\begin{theorem}
Suppose $V$ is multiple of $4$ and $\dfrac{V}{4}$ has prime factorization $\displaystyle 2^x 3^y \prod_{i = 1}^z p_i^{k_i} \prod_{j = 1}^w q_j^{\ell_j}$ where the $p_i$ are distinct primes congruent to $1 \pmod 3$,  the $q_j$ are distinct primes greater than $2$ that are congruent to $2 \pmod 3$, $x, y, z, w \geq 0$, $k_i \geq 1$ for $1 \leq i \leq z$, and $\ell_j \geq 1$ for $1 \leq j \leq w$. 

The number of graph isomorphism classes for $(3,6)$-fullerenes with $V$ vertices is given by the following formulas:

\begin{enumerate}
\item If $x = 0$ and all $\ell_j$ are even, 
\begin{align*}\displaystyle \gamma(V) &=   \dfrac{3^{y+1} - 1}{12} \prod_{i = 1}^z \dfrac{p_i^{k_i + 1} - 1}{p_i - 1} \prod_{j = 1}^w \dfrac{q_j^{\ell_j + 1} - 1}{q_j - 1}   \\
&+ 
 \dfrac{1}{3}\prod_{i = 1}^{z} (k_i + 1)  + \dfrac{1}{2} (y+1)
 \prod_{i = 1}^{z} (k_i + 1) \prod_{j = 1}^w (\ell_j + 1) 
\end{align*}

\item If $x > 0$, $x$ is even, and all $\ell_j$ are even, 

\begin{align*}
\displaystyle \gamma(V)  &=  \dfrac{(2^{x+1} - 1)(3^{y+1} - 1)}{12} \prod_{i = 1}^z \dfrac{p_i^{k_i + 1} - 1}{p_i - 1} \prod_{j = 1}^w \dfrac{q_j^{\ell_j + 1} - 1}{q_j - 1}    \\ 
 &+  \dfrac{1}{3} 
 \prod_{i = 1}^{z} (k_i + 1)  + \dfrac{(2x - 1)}{2}(y+1)\prod_{i = 1}^{z} (k_i + 1)\prod_{j = 1}^w (\ell_j + 1) 
 \end{align*}
\item If $x = 0$ and $\ell_j$ is odd for some $1 \leq j \leq w$, 
\begin{align*}
    \displaystyle \gamma(V) &=  \dfrac{(3^{y+1} - 1)}{12} \prod_{i = 1}^z \dfrac{p_i^{k_i + 1} - 1}{p_i - 1} \prod_{j = 1}^w \dfrac{q_j^{\ell_j + 1} - 1}{q_j - 1}     \\
&+ \dfrac{1}{2}(y+1)\prod_{i = 1}^{z} (k_i + 1)\prod_{j = 1}^w (\ell_j + 1) 
\end{align*}

\item If $x > 0$ and either $x$ is odd or $\ell_j$ is odd  for some $1 \leq j \leq w$, then
\begin{align*} \gamma(V) &=  \dfrac{(2^{x+1} - 1)(3^{y+1} - 1)}{12} \prod_{i = 1}^z \dfrac{p_i^{k_i + 1} - 1}{p_i - 1} \prod_{j = 1}^w \dfrac{q_j^{\ell_j + 1} - 1}{q_j - 1}     \\
&+ \dfrac{(2x - 1)}{2}(y+1)\prod_{i = 1}^{z} (k_i + 1)\prod_{j = 1}^w (\ell_j + 1) 
\end{align*}
\end{enumerate}
\label{thm-gamma}
\end{theorem}

\begin{proof}
These formulas follow directly from Lemma~\ref{lem-$(3,6)$-fullereneGraphCount}, the formula for $\delta(V)$  in Theorem~\ref{thm-delta}, the formula for $\mu(V)$  in Proposition~\ref{prop-mu}, and the formula for $\sigma(V)$ in Equation~\ref{eqn-sigma}.

\end{proof}

We will briefly consider the number of graph isomorphism classes for $(3,6)$-fullerenes with mirror symmetry, 3-fold rotational symmetry, and both types of symmetry. The number of graph isomorphism classes of $(3,6)$-fullerenes with mirror symmetry is the same as the number of $(3,6)$-fullerenes with mirror symmetry, and is given in Proposition~\ref{prop-mu}. Similarly, the number of graph isomorphism classes of $(3,6)$-fullerenes with both 3-fold rotational symmetry and mirror symmetry is given in Proposition~\ref{prop-nu}.

The number of graph isomorphism classes for $(3,6)$-fullerenes with 3-fold rotational symmetry follows directly from the count of $(3,6)$-fullerenes with 3-fold rotational symmetry:

\begin{corollary}
Suppose $V$ is a multiple of $4$ and $\frac{V}{4}$  has prime factorization $\displaystyle 3^y \prod_{i = 1}^z p_i^{k_i} \prod_{j = 1}^{w} q_j^{\ell_j}$, with distinct primes $p_i \equiv 1 \pmod 3$ and distinct primes $q_j \equiv 2 \pmod 3 $, where $z, w,$ and the exponent $y$ can be any integer $\geq 0$, and the other exponents are all $\geq 1$. Then the number of graph isomorphism classes for $(3,6)$-fullerenes with $V$ vertices and 3-fold rotational symmetry  is 0 if $\ell_j$ is odd for any $j$ with $1 \leq j \leq w$. If all $\ell_j$ are even, but some $k_i$ are odd, then this number is  $\displaystyle \dfrac{1}{2} \prod_{i = 1}^z (k_i + 1)$. If all $k_i$ and $\ell_j$ are even, then this number is $\displaystyle \dfrac{1}{2} \prod_{i = 1}^z (k_i + 1) + \dfrac{1}{2}$.
\label{cor-graphIsoWith3FoldSymm}
\end{corollary}

\begin{proof}
For $(3,6)$-fullerenes with mirror symmetry, each
$(3,6)$-fullerene corresponds to one graph isomorphism class, while for $(3,6)$-fullerenes without mirror symmetry, two $(3,6)$-fullerenes correspond to the same graph isomorphism class. Therefore, the number of $(3,6)$-fullerene graph isomorphism classes for $(3,6)$-fullerenes with $V$ vertices with 3-fold rotational symmetry is $\dfrac{1}{2} (\delta(V) - \nu(V)) + \nu(V)$, which simplifies to $\dfrac{1}{2} \delta(V) + \dfrac{1}{2} \nu(V)$. 
The given formulas now follow from the formula for $\delta(V)$  in Theorem~\ref{thm-delta} and the formula for $\nu(V)$ in Proposition~\ref{prop-nu}.

\end{proof}

\section{Tabulated values and examples}
\label{sec-examples}

Recall that $\sigma(V)$ is the number of signatures of $(3,6)$-fullerenes with $V$ vertices, $\alpha(V)$ is the number of $(3,6)$-fullerenes with $V$ vertices, $\delta(V)$ is the number of $(3, 6)$-fullerenes with 3-fold symmetry, $\mu(V)$ is  the number of $(3,6)$-fullerenes with mirror symmetry,  $\nu(V)$ is the number of $(3,6)$-fullerenes with both 3-fold rotational symmetry and mirror symmetry, and $\gamma(V)$ is the number of graph isomorphism classes of $(3,6)$-fullerenes with $V$ vertices. Tables~\ref{table-sample1} and \ref{table-sample2} show the computations of these quantities from the appropriate formulas for $V =   4  \cdot 3 \cdot 5^2 \cdot 7^2 \cdot 13 = 191,100$ and for $V = 4 \cdot 2^2 \cdot 5^2 \cdot 11^2 = 48,400$.

\begin{table}[h]
\begin{tabular}{|c|c|c|c|} \hline
Function & Computation & Numerical Answer & Formula Reference \\ \hline
$\sigma(191,100)$ & $\frac{3^2 - 1}{3 - 1}\cdot \frac{5^3 -1}{5 - 1} \cdot \frac{7^3 - 1}{7-1} \cdot \frac{13^2 - 1}{13 - 1}$ & $98,952$ & Equation~\ref{eqn-sigma} \\ \hline
$\delta(191,100)$ & $(2 +1)(1+1)$ & $6$ & Theorem~\ref{thm-delta} \\ \hline
$\alpha(191,100)$ & $\frac{1}{3} \sigma(191,100) + \frac{2}{3} \delta(191,100)$ & $32,988$ & Theorem~\ref{thm-$(3,6)$-fullereneCount} \\ \hline 
$\mu(191,100)$ & $(1+1)(2+1)(2+1)(1+1)$ & $36$ & Proposition~\ref{prop-mu} \\ \hline
$\nu(191,100)$ & 0 & 0 & Proposition~\ref{prop-nu} \\ \hline
$\gamma(191,100)$ & $\frac{1}{6} \sigma(191,100) + \frac{1}{3} \delta(191,100) + \frac{1}{2} \mu(191,100)$ & $16,512$ & Lemma~\ref{lem-$(3,6)$-fullereneGraphCount} \\ \hline
\end{tabular}

\caption{Sample computation for $V =  4  \cdot 3 \cdot 5^2 \cdot 7^2 \cdot 13 = 191,100$ vertices}

\label{table-sample1}
\end{table}

\begin{table}[h]
\begin{tabular}{|c|c|c|c|} \hline
Function & Computation & Numerical Answer & Formula Reference \\ \hline
$\sigma(48,400)$ & $\frac{2^3 - 1}{2 - 1}\cdot \frac{5^3 -1}{5 - 1} \cdot \frac{11^3 - 1}{11-1}$ & $28,861$ & Equation~\ref{eqn-sigma} \\ \hline
$\delta(48,400)$ & $1$ & $1$ & Theorem~\ref{thm-delta} \\ \hline
$\alpha(48,400)$ & $\frac{1}{3} \sigma(48,400) + \frac{2}{3} \delta(48,400)$ & $9,621$ & Theorem~\ref{thm-$(3,6)$-fullereneCount} \\ \hline 
$\mu(48,400)$ & $(2\cdot 2 - 1)(2+1)(2+1)$ & $27$ & Proposition~\ref{prop-mu} \\ \hline
$\nu(48,400)$ & 1 & 1 & Proposition~\ref{prop-nu} \\ \hline
$\gamma(48,400)$ & $\frac{1}{6} \sigma(48,400) + \frac{1}{3} \delta(48,400) + \frac{1}{2} \mu(48,400)$ & $4824$ & Lemma~\ref{lem-$(3,6)$-fullereneGraphCount} \\ \hline
\end{tabular}
\caption{Sample computation for $V = 4 \cdot 2^2 \cdot 5^2 \cdot 11^2=48,400$ vertices}
\label{table-sample2}
\end{table}

Counts of $(3,6)$-fullerenes and $(3,6)$-fullerene graph isomorphism classes for up to 480 vertices are given in Table~\ref{table:bigCounts}. These numbers replicate the counts given in  Table~2 of \cite{green2024polyhedra}. They also agree with the counts from the $N_3$ columns of Table~5 of \cite{deza2005zigzag}, but close attention needs to be paid to the details. The counts in columns $N_3$ in Table~5 of \cite{deza2005zigzag} are of graph isomorphism classes and do not include godseyes, the $(3,6)$-fullerenes with spines of length 0 and signatures of the form $(0, b, 0)$ mentioned in Section~\ref{sec-intro}. Our counts  in the $\gamma(V)$ columns of Table~\ref{table:bigCounts} are also of graph isomorphism classes but do include godseyes. Since there is exactly one godseye for each number of vertices that is a multiple of 4, the counts in the $\gamma(V)$ columns of Table~\ref{table:bigCounts} are exactly one higher than the counts of $N_3$ in Table 5 of \cite{deza2005zigzag}.

\begin{table}
\centering
\begin{tabular}{|c|c|c||c|c|c||c|c|c||c|c|c|c|c|c|}\hline
$V$ & $\alpha(V)$ & $\gamma(V$) & 
$V$ & $\alpha(V)$ & $\gamma(V)$ &
$V$ & $\alpha(V)$ & $\gamma(V)$&
$V$ & $\alpha(V)$ & $\gamma(V)$ \\	\hline
4	&	1	&	1	&	124	&	12	&	7	&	244	&	22	&	12	&	364	&	40	&	22	\\	\hline
8	&	1	&	1	&	128	&	21	&	15	&	248	&	32	&	17	&	368	&	56	&	31	\\	\hline
12	&	2	&	2	&	132	&	16	&	10	&	252	&	36	&	21	&	372	&	44	&	24	\\	\hline
16	&	3	&	3	&	136	&	18	&	10	&	256	&	43	&	27	&	376	&	48	&	25	\\	\hline
20	&	2	&	2	&	140	&	16	&	10	&	260	&	28	&	16	&	380	&	40	&	22	\\	\hline
24	&	4	&	3	&	144	&	31	&	20	&	264	&	48	&	26	&	384	&	84	&	51	\\	\hline
28	&	4	&	3	&	148	&	14	&	8	&	268	&	24	&	13	&	388	&	34	&	18	\\	\hline
32	&	5	&	5	&	152	&	20	&	11	&	272	&	42	&	24	&	392	&	57	&	30	\\	\hline
36	&	5	&	4	&	156	&	20	&	12	&	276	&	32	&	18	&	396	&	52	&	29	\\	\hline
40	&	6	&	4	&	160	&	30	&	20	&	280	&	48	&	26	&	400	&	73	&	41	\\	\hline
44	&	4	&	3	&	164	&	14	&	8	&	284	&	24	&	13	&	404	&	34	&	18	\\	\hline
48	&	10	&	8	&	168	&	32	&	18	&	288	&	65	&	40	&	408	&	72	&	38	\\	\hline
52	&	6	&	4	&	172	&	16	&	9	&	292	&	26	&	14	&	412	&	36	&	19	\\	\hline
56	&	8	&	5	&	176	&	28	&	17	&	296	&	38	&	20	&	416	&	70	&	40	\\	\hline
60	&	8	&	6	&	180	&	26	&	16	&	300	&	42	&	24	&	420	&	64	&	36	\\	\hline
64	&	11	&	9	&	184	&	24	&	13	&	304	&	48	&	27	&	424	&	54	&	28	\\	\hline
68	&	6	&	4	&	188	&	16	&	9	&	308	&	32	&	18	&	428	&	36	&	19	\\	\hline
72	&	13	&	8	&	192	&	42	&	28	&	312	&	56	&	30	&	432	&	94	&	53	\\	\hline
76	&	8	&	5	&	196	&	21	&	12	&	316	&	28	&	15	&	436	&	38	&	20	\\	\hline
80	&	14	&	10	&	200	&	31	&	17	&	320	&	62	&	38	&	440	&	72	&	38	\\	\hline
84	&	12	&	8	&	204	&	24	&	14	&	324	&	41	&	23	&	444	&	52	&	28	\\	\hline
88	&	12	&	7	&	208	&	34	&	20	&	328	&	42	&	22	&	448	&	84	&	49	\\	\hline
92	&	8	&	5	&	212	&	18	&	10	&	332	&	28	&	15	&	452	&	38	&	20	\\	\hline
96	&	20	&	15	&	216	&	40	&	22	&	336	&	76	&	44	&	456	&	80	&	42	\\	\hline
100	&	11	&	7	&	220	&	24	&	14	&	340	&	36	&	20	&	460	&	48	&	26	\\	\hline
104	&	14	&	8	&	224	&	40	&	25	&	344	&	44	&	23	&	464	&	70	&	38	\\	\hline
108	&	14	&	9	&	228	&	28	&	16	&	348	&	40	&	22	&	468	&	62	&	34	\\	\hline
112	&	20	&	13	&	232	&	30	&	16	&	352	&	60	&	35	&	472	&	60	&	31	\\	\hline
116	&	10	&	6	&	236	&	20	&	11	&	356	&	30	&	16	&	476	&	48	&	26	\\	\hline
120	&	24	&	14	&	240	&	56	&	34	&	360	&	78	&	42	&	480	&	120	&	70	\\	\hline
\end{tabular}
    \caption{For each possible number of vertices $V \leq 480$, $\alpha(V)$ is the number of $(3,6)$-fullerenes (where left and right handed versions of chiral $(3,6)$-fullerenes are considered distinct) and $\gamma(V)$ is the number $(3,6)$-fullerene graph isomorphism classes (where left and right handed versions of chiral $(3,6)$-fullerenes are considered the same).}
    \label{table:bigCounts}
\end{table}

\section*{Acknowledgements}
The authors thank the reviewer for excellent suggestions.

\section*{Competing interests}
The authors have no relevant financial or non-financial interests to disclose.


\pagebreak

\bibliographystyle{plain}
\bibliography{3-6fullereneExactCount}

\end{document}